\documentclass[leqno]{amsart}

\usepackage{amssymb,amsfonts,amsthm}%,psfig}
\usepackage{graphicx}
\usepackage{subfigure}
\usepackage{amsthm}
\usepackage{tabularx}
\usepackage{amsmath}
\usepackage{bm}             %% boldmath-command: \bm{}
\usepackage[mathscr]{eucal}
\usepackage{color} 
\usepackage{cancel}
\usepackage{enumerate}
\usepackage{psfrag}
\usepackage{bbm}
\usepackage{appendix}
\usepackage[colorlinks=true,linkcolor=red,citecolor=blue]{hyperref}

\usepackage[normalem]{ulem}
\usepackage{cleveref}

\newcommand{\R}{\mathbb{R}}
\newcommand{\bfR}{\mathbb{R}}
\newcommand{\N}{\mathbb{N}}

\newcommand{\mcA}{\mathcal{A}}
\newcommand{\mcM}{\mathcal{M}}
\newcommand{\mcS}{\mathcal{S}}
\newcommand{\mcD}{\mathcal{D}}
\newcommand{\mcG}{\mathcal{G}}

\newcommand{\mcV}{\mathcal{V}^\alpha}

\newcommand{\bS}{\mcS^\alpha}

\newcommand{\infnorm}[1]{\left\| #1 \right\|_{\infty}} %(Grace added)
\newcommand{\rndP}[1]{\left(#1 \right)} %(Grace added)
\newcommand{\sqrP}[1]{\left[#1 \right]} %(Grace added)

\newtheorem{theorem}{Theorem}[section]
\newtheorem*{theorem*}{Theorem}

\newtheorem{lemma}[theorem]{Lemma}

\theoremstyle{definition}

\newtheorem{remark}{Remark}[section]

\usepackage{fullpage}

\title{Arbitrary Polynomial Decay Rates of Neutral, Collisionless Plasmas}
\author{Grace Mattingly, Stephen Pankavich}
\address{Department of Applied Mathematics and Statistics, Colorado School of Mines, Golden, CO 80401.}
\email{gmattingly@mines.edu}
\email{pankavic@mines.edu}

\author{Jonathan Ben-Artzi}
\address{School of Mathematics, Cardiff University, Cardiff, Wales, UK.}
\email{ben-artzij@cardiff.ac.uk}

\date{\today}
\thanks{The first author was supported by an NSF Graduate Research Fellowship. The second author was supported in part by NSF grants DMS-1911145 and DMS-2107938.}

\begin{document}

%TO DO:
%Look at Section 4

\maketitle

\begin{abstract}
A multispecies, collisionless plasma is modeled by the Vlasov-Poisson system. 
Assuming the plasma is neutral and the electric field decays with sufficient rapidity as $t \to\infty$, we show that solutions can be constructed with arbitrarily fast, polynomial rates of decay, depending upon the properties of the limiting spatial average and its derivatives.
In doing so, we establish, for the first time, a countably infinite number of asymptotic profiles for the charge density, electric field, and their derivatives, each of which is necessarily realized by a sufficiently smooth solution and exceeds the established dispersive decay rates.
Finally, in each case we establish a linear $L^\infty$ scattering result for every particle distribution function, namely we show that they converge as $t \to \infty$ along the transported spatial characteristics at increasingly faster rates.
\end{abstract}

\section{Introduction}
\label{sec:intro}
We consider a plasma comprised of a large number of charged particles. If there are $N\in\N$ distinct species of particles within the plasma, then those of the $\alpha$th species (for $\alpha = 1, ..., N$)
have charge $q_\alpha \in \mathbb{R}$, mass $m_\alpha > 0$, and are distributed in phase space at time $t \geq 0$ according to the function $f^\alpha(t,x,v)$ where $x \in \bfR^3$ represents particle position and $v\in \bfR^3$ particle velocity. 
Assuming that electrostatic forces dominate collisional effects, the time evolution of the plasma is described by the multispecies Vlasov-Poisson system
\begin{equation}
\tag{VP}
\label{VP}
\left \{ \begin{aligned}
& \partial_{t}f^\alpha+v\cdot\nabla_{x}f^\alpha+ \frac{q_\alpha}{m_\alpha} E \cdot\nabla_{v}f^\alpha=0, \qquad \alpha = 1, ..., N,\\
& \rho(t,x)= \sum_{\alpha=1}^N q_\alpha \int_{\mathbb{R}^3} f^\alpha(t,x,v) \,dv\\
& E(t,x) = \nabla_x ( \Delta_x)^{-1} \rho(t,x) = \frac{1}{4\pi}\int_{\mathbb{R}^3} \frac{x-y}{\vert x - y \vert^3} \rho(t,y) \ dy
\end{aligned} \right .
\end{equation}
with prescribed initial conditions $f^\alpha(0,x,v) = f^\alpha_0(x,v)\geq0$ for each $\alpha = 1, ..., N$.
Here, $E(t,x)$ represents the electric field induced by the charged particles, and $\rho(t,x)$ is the associated density of charge within the plasma. A quantity that will play a significant role throughout this paper is the total net charge
	\begin{equation}
	\label{eq:M}
	\mcM
	:=
	\int\rho(0,x)\ dx = \sum_{\alpha = 1}^N \iint q_\alpha f_0^\alpha(x,v) \ dv dx
	\end{equation}
which, in fact, is preserved at later times $t>0$.

%Such global existence results often depend upon either the propagation of (spatial, velocity, or transported) moments or precise estimates for the growth of the characteristics associated to \eqref{VP}, which are defined by
%\begin{equation*}
%\label{char}
%\left \{
%\begin{aligned}
%&\dot{\mcX}(t, \tau, x, v)=\mcV(t, \tau, x, v)\\
%&\dot{\mcV}(t, \tau, x, v)= \frac{q_\alpha}{m_\alpha} E(t, \mcX(t, \tau, x, v))
%\end{aligned}
%\right.
%\end{equation*}
%with initial conditions
%$\mcX(\tau, \tau, x, v) = x$ and
%$\mcV(\tau, \tau, x, v) = v$
%for $\alpha = 1, ..., N$.

It is well-known that given smooth initial data with either compact support in phase space or finite moments, \eqref{VP} possesses a  global-in-time smooth solution \cite{LP, Pfaf, Schaeffer}.
Though well-posedness has been intensely studied, the large time behavior of solutions to \eqref{VP} has only recently become better understood. Partial results for the Cauchy problem are known in some special cases, including small data \cite{BD, CK, HRV, Ionescu, Smulevici}, monocharged and spherically-symmetric data \cite{Horst, Pankavich2020}, and lower-dimensional settings \cite{BKR, BMP, GPS, GPS2, GPS4, Sch}.
Additionally, an understanding of the intermediate asymptotic behavior was obtained in \cite{BCP1}, namely that there are solutions for which the $L^\infty$ norms of the charge density and electric field can be made arbitrarily large at some later, finite time regardless of their initial size.
For general background concerning \eqref{VP} and associated kinetic equations, we refer the reader to \cite{Glassey, Rein}. 

Due to the dispersive properties imparted upon the system by the transport operator $\partial_t + v \cdot \nabla_x$ and the velocity averaging inherent to the observables, one generally expects that the charge density and electric field decay to zero like $t^{-3}$ and $t^{-2}$, respectively, as $t \to \infty$ for all smooth solutions of \eqref{VP}. 
In fact, both small data and spherically-symmetric solutions have been shown to exhibit exactly this behavior, and it is known, at least in lower dimensions, that the Cauchy problem does not possess smooth steady states \cite{GPS5}.
In the periodic case ($x \in \mathbb{T}^d$), phase mixing can further drive these quantities to decay exponentially fast to zero, depending upon the regularity of the solution, by creating filamentation through a process known as Landau damping \cite{VM}, and this has been shown for perturbations of a neutral plasma around a spatially-homogeneous equilibrium on the torus.
In contrast, the results concerning Landau damping for \eqref{VP} with $x \in \bfR^3$ remain somewhat limited \cite{Ionescu2} and do not prove decay of the maximal electric field faster than $t^{-2}$.
Indeed, as shown in \cite[Theorem 1.1] {Pankavich2022}, non-neutral plasmas ($\mcM \neq 0$) cannot experience decay that is faster than the rates provided by the dispersive mechanisms in the system.
Though, via Theorem \ref{oldT2}, this same paper determined that faster rates can occur in a neutral plasma  ($\mcM = 0$) should the limiting charge density vanish.

The current paper significantly extends this idea to demonstrate that solutions attain decay rates of \emph{any} polynomial order greater than those attributed to the dispersive properties of the system.
Namely, we show for the first time that neutral plasmas may attain decay rates of any integer order between the dispersive regime ($E \sim t^{-2}$, $\rho \sim t^{-3}$) and the phase mixing rates ($E, \rho \sim e^{-\gamma t}, \gamma > 0$),
which closes the gap between the behavior attributed to solutions in the whole space in comparison to those of the periodic or screened \cite{BMM} system.
In this way, we demonstrate that the system allows for infinitely many different regimes of asymptotic behavior, though the scattering behavior of the distribution function can only assume two distinct states (linear and modified).

Of course, the dynamical behavior of the system depends intrinsically upon establishing decay of the electric field. 
To date, the best \emph{a priori} rate known \cite{Yang} for the electric field stemming from arbitrary solutions of \eqref{VP} is
$$\| E(t) \|_\infty \leq C(1+t)^{-1/6},$$
and this is derived from precise estimates of the growth of the maximal velocity on the support of $f(t)$, from which the estimate of $\|E(t)\|_\infty$ is deduced. 
Unfortunately, this resulting estimate is far from what is believed to be the optimal decay rate.
Additionally, while maximal velocity support estimates have been beneficial to improving the field decay rate \cite{Chen2, Pallardsupport, Jacksupport}, even a sharp estimate (i.e, a uniform bound) on the support cannot allow one to conclude a sharp decay rate of the field.
So, it appears that we are quite far from obtaining precise \emph{a priori} estimates of the field.
%
%Still, even when the field decays at a fast rate, the dynamics of the remaining quantities in \eqref{VP} are not well-understood.
Therefore, the goal of the current work is to establish the precise large-time dynamical behavior of solutions to \eqref{VP} whenever the electric field is known to decay with sufficient rapidity and demonstrate the wide variety of asymptotic behavior displayed by solutions.  \\

\textbf{Organization of the paper.}
Section \ref{sec:intro} is dedicated to the introduction of the problem and the statement of our main results, as well as an overview of previous results and the strategy of the proofs.
In Section \ref{Lemmas}, we prove some preliminary lemmas that provide asymptotic bounds (as $t\to\infty$) on the charge density $\rho(t,x)$ and electric field $E(t,x)$, as well as their derivatives.
Then, these bounds are used to prove the two main theorems  in Section \ref{sec:proofs}.
A rather long argument establishing asymptotic bounds for derivatives of $g^\alpha$ is postponed until  Section \ref{sec:bounds-derivs-g} in order to facilitate the exposition.

\subsection{Overview}

Due to the global existence theorem, all quantities of interest are bounded for finite time; thus, we are only concerned with large time estimates.
Hence, we use the notation
$$A(t) \lesssim B(t)$$
to represent the statement that there is a constant $C > 0$, independent of $t$, 
such that
$A(t) \leq CB(t)$
for all $t$ sufficiently large.
Furthermore, the notation 
$$A(t) \sim B(t)$$
indicates
$$A(t) \lesssim B(t) \qquad \mathrm{and} \qquad B(t) \lesssim A(t).$$
%When necessary, $C$ will denote a positive constant (independent of the solution) that may change from line to line.

\textbf{Field decay assumption.} We assume that the dispersive effects or other physical phenomena, such as charge cancellation, in the system induce a strong decay of the electric field, namely that 
\begin{equation}
\label{Assumption}
\tag{A}
\exists\ p>\frac53 \text{ such that:}\quad\|E(t) \|_\infty \lesssim t^{-p}.
\end{equation}
We note that this assumption is known to be satisfied for monocharged, spherically-symmetric initial data \cite{Horst, Pankavich2020} and for all previously constructed perturbative solutions (e.g., \cite{BD, Ionescu}) regardless of the number or type of species involved.\\

Though we will assume compactly-supported initial data, this may not be necessary as velocity, spatial, and transported moments \cite{Castella, Chen2, LP, Pallardspatial, Pallardsupport} have all been used in lieu of this assumption to develop the existence theory.
Due to the Taylor series expansion we will employ later, the use of transported moments would be very natural in the context of understanding the large time behavior.
In addition, the regularity assumptions on initial data may be weakened to arrive at similar convergence results in weaker topologies (see \cite{Ionescu}).
The novelty herein is that solutions are obtained with faster field and charge density decay rates than previously exhibited in the whole space, and the precise asymptotic profile of the electric field, its derivatives, and the charge density are obtained for each associated solution.  
Additionally, our methods display the crucial dependence on the spatial support and velocity derivatives of the distribution function when translated along free-streaming spatial characteristics.
Finally, our results apply directly to general conditions that may be satisfied by any neutral plasma and not merely to small data solutions.

\subsection{Review of Previous Results}
To facilitate the presentation of our novel theorems, we first quote some  recent results which serve as a starting point for our findings.
\subsubsection{Modified Scattering}
Under the assumption \eqref{Assumption}, we first summarize the previous results of \cite{Pankavich2022} that the scaled charge density and field converge to limits based upon the limiting spatial average, namely
\begin{theorem}[\cite{{Pankavich2022}}]
\label{oldT1}
Consider a solution $f^\alpha\in C^2((0,\infty)\times\R^6)$ of \eqref{VP} with initial data $f^\alpha_0 \in C_c^2(\bfR^6)$. Assume that \eqref{Assumption} holds. Then, we have the following:
\begin{enumerate}[(a)]
%\item
% For every $\alpha = 1, ..., N$, $\tau \geq 0$ and $(x,v) \in \mcU$ the limiting function $\mcV_\infty$ defined by
%$$\mcV_\infty(\tau, x, v) :=  \lim_{t \to \infty} \mcV(t, \tau, x, v)$$ 
%exists and is $C^2$ and bounded.
%Additionally, for $\tau \geq 0$ and $(x,v) \in \mcU$,
%$$\left | \mcV(t, \tau, x, v) - \mcV_\infty(\tau, x, v)  \right | \lesssim t^{-1}.$$

\item 
%For every $\alpha = 1, ..., N$ define
%$$\Omv = \left \{ \mcV_\infty(0, x, v) : (x, v) \in \bS_f(0) \right \}.$$
%Then, 
For each $\alpha = 1, ..., N$ there exists $F^{\alpha,0}_\infty \in C_c^1(\bfR^3)$ 
such that the spatial average
$$F^{\alpha,0}(t,v) = \int f^\alpha(t,x, v) \ dx$$
satisfies $F^{\alpha,0}(t,v) \to F^{\alpha,0}_\infty(v)$ uniformly as $t \to \infty$ with
$$\| F^{\alpha,0}(t) - F^{\alpha,0}_\infty \|_\infty \lesssim t^{-1}\ln^{4}(t).$$
Moreover, the net density
$$\rho_{0}(t,x) = \sum_{\alpha=1}^N  q_\alpha F^{\alpha,0}\left (t,\frac{x}{t} \right)$$
converges at the same rate 
to
$$\rho_{0,\infty}\left (\frac{x}{t} \right ) = \sum_{\alpha=1}^N  q_\alpha F^{\alpha,0}_\infty\left (\frac{x}{t} \right ) ,$$
which satisfies
\begin{equation}
\label{Pinfmass}
\int \rho_{0,\infty}(v) \ dv = \mcM,
\end{equation}
where $\mcM$ is given by \eqref{eq:M}.

\item Define $E_{0,\infty}(v) = \nabla_v(\Delta_v)^{-1} \rho_{0,\infty}(v)$. Then, we have the self-similar asymptotic profiles
\begin{align*}
\sup_{x\in \bfR^3} \left | t^2 E(t,x) - E_{0,\infty} \left(\frac{x}{t} \right) \right | & \lesssim t^{-1}\ln^{4}(t),\\
\sup_{x \in \bfR^3}  \left | t^3 \nabla_xE(t,x) - \nabla_v E_{0,\infty} \left(\frac{x}{t} \right) \right | & \lesssim t^{-1} \ln^{6}(t),\\
\sup_{x \in \bfR^3}   \left | t^3 \rho(t,x) - \rho_{0,\infty} \left(\frac{x}{t} \right) \right | & \lesssim t^{-1} \ln^{5}(t).
\end{align*}

\item
For each $\alpha = 1, ..., N$ there is $f^\alpha_\infty \in C_c^1(\bfR^6)$ 
such that
$$f^\alpha \left(t,x +vt - \frac{q_\alpha}{m_\alpha} \ln(t)E_{0,\infty}(v),v \right) \to f^\alpha_\infty(x,v)$$
uniformly
as $t \to \infty$, namely we have the convergence estimate
$$\sup_{(x,v) \in \bfR^6} \left | f^\alpha \left(t,x +vt - \frac{q_\alpha}{m_\alpha} \ln(t)E_{0,\infty}(v), v \right) - f^\alpha_\infty(x,v) \right |  \lesssim t^{-1}\ln^{4}(t).$$
Furthermore, for all $\alpha = 1, ..., N$ we have
$$ F^{\alpha,0}_\infty(v) = \int f^\alpha_\infty(x,v) \ dx.$$
\end{enumerate}
\end{theorem}

If the plasma is non-neutral, i.e. $\mcM \neq 0$, then $\rho_{0,\infty} \not\equiv 0$ due to \eqref{Pinfmass} and these estimates are sharp, up to a correction in the logarithmic powers of the error terms.
Hence, $\| \rho(t) \|_\infty, \| \nabla_x E(t) \|_\infty \sim \mathcal{O}\left(t^{-3} \right)$ and $\| E(t) \|_\infty \sim \mathcal{O}\left(t^{-2} \right)$.
However, when the plasma is neutral, i.e. $\mcM = 0$, it is possible that the limiting density $\rho_{0,\infty}$ (and hence, the limiting field $E_{0,\infty}$) is identically zero, which implies stronger decay of $\| \rho(t) \|_\infty$, $\| E(t) \|_\infty$, $\| \nabla_x E(t) \|_\infty$, and related quantities.
Indeed, as we will show these quantities can decay at any polynomial rate greater than the above powers, depending upon the behavior of the limiting charge density. Note that a nontrivial neutral plasma is necessarily comprised of more than one species of charged particles, so that $N\geq2$ in what follows.

\begin{theorem}[\cite{{Pankavich2022}}]
\label{oldT2}
Under the assumptions of Theorem \ref{oldT1}, if additionally $\mcM = 0$ and $\rho_{0,\infty} \equiv 0$, then the asymptotic behavior described above is altered in the following manner:
\begin{enumerate}[(a)]
%\item 
%For any $\alpha = 1, ..., N$, $\tau \geq 0$ and $(x,v) \in \mcU$, we have
%$$\left | \mcV(t, \tau, x, v) - \mcV_\infty(\tau, x, v)  \right | \lesssim t^{-2}.$$

\item We have the improved estimates
\begin{align*}
\| F^{\alpha,0}(t) - F^{\alpha,0}_\infty \|_\infty & \lesssim t^{-2}, \quad \alpha = 1, ..., N, &
%\| \mcP(t) \|_\infty & \lesssim t^{-2}, \\
\| E(t) \|_\infty & \lesssim t^{-3},\\
 %\mu(t) & \lesssim 1, \\
\| \nabla_xE(t) \|_\infty & \lesssim t^{-4} \ln(t), &
\| \rho(t) \|_\infty & \lesssim t^{-4}, &  
%\| j(t) \|_\infty & \lesssim t^{-4}.
\mcG_v^1(t) +  \mcG_{x,v}^1(t) & \lesssim 1.
\end{align*}

%\item For any $\alpha = 1, ..., N$, the spatial characteristics
%$$\mcY(t,\tau, x, v) = \mcX(t, \tau, x, v) - t \mcV(t, \tau, x, v),$$ 
%converge as $t \to \infty$, and the limiting functions 
%$$\mcY_\infty(\tau, x, v) :=  \lim_{t \to \infty} \mcY(t, \tau, x, v)$$
%are $C^2$ and bounded.
%Additionally, for any $\tau \geq 0$ and $(x,v) \in \mcU$,\
%$$|\mcY(t, \tau, x, v) - \mcY_\infty(\tau, x, v) | \lesssim t^{-1}.$$

\item
%For any $\alpha = 1, ..., N$ define
%$$\Omy = \left \{ \mcY_\infty(0, x, v) : (x, v) \in \bS_f(0) \right \}$$
%and
%$\Om^\alpha = \Omy \times \Omv$.
%
The distribution functions scatter linearly, namely
for each $\alpha = 1, ..., N$ there is $f^\alpha_\infty \in C_c^1(\bfR^6)$  such that
$$f^\alpha(t,x +vt,v) \to f^\alpha_\infty(x,v)$$
uniformly
as $t \to \infty$ with the convergence estimate
$$\sup_{(x,v) \in \bfR^6} \left | f^\alpha(t,x +vt, v) - f^\alpha_\infty(x,v) \right |  \lesssim t^{-1}.$$
\end{enumerate}
\end{theorem}

%\begin{remark}
%The estimates on $\mu(t)$ and $\mcG_v^1(t) +  \mcG_{x,v}^1(t)$ are not explicitly stated within Theorem 1.2 of \cite{{Pankavich2022}}.
%However, the uniform bounds on both the spatial support and the spatial and velocity derivatives of $g^\alpha(t,x,v)$ follow by applying the stated field estimate to Lemmas 2.3, 2.6, and 2.8 therein.
%\end{remark}

Of course, the upper bounds listed above do not guarantee a sharp rate of decay or an identification of the correct asymptotic behavior of these quantities.
For instance, $\rho(t,x)  = E(t,x) \equiv 0$ satisfy these estimates and can be constructed from a rudimentary initial profile in a neutral plasma, for example, in which the distributions of all positive and negative charges overlap everywhere in phase space.
Therefore, we now turn our attention to stating new results that every order of decay is actually attained by some solution of \eqref{VP}.

\subsubsection{Polyhomogeneous Expansions}
Recently, a growing interest has emerged in understanding the precise asymptotic behavior of small data solutions to \eqref{VP}, e.g. in \cite{Bigorgne-Ruiz,Ionescu,Schlue-Taylor}. More specifically, authors have focused on studying expansions of the charge density and field that include higher order terms beyond $\rho_{0,\infty}$ and $E_{0,\infty}$ defined above. As in  previous works, one key tool is a change of reference frame, defining
	\begin{equation}
	\label{eq:g}
	g^\alpha(t,x,v) = f^\alpha(t, x+vt, v)
	\end{equation}
for $\alpha = 1, ..., N$. In this new frame the charge density and field can be expressed, respectively, as
\begin{equation}
\label{rho}
\rho(t,x)= t^{-3} \sum_{\alpha=1}^N q_\alpha  \int_{\mathbb{R}^3} g^\alpha \left(t,y,\frac{x-y}{t} \right)\,dy,
\end{equation}
and
\begin{equation}
\label{E}
E(t,x) = \frac{1}{4\pi t^2} \sum_{\alpha=1}^N q_\alpha \iint \frac{\xi}{|\xi|^3} \ g^\alpha \left(t, y, \frac{x-y}{t} - \xi \right) \ dyd\xi.
\end{equation}
To arrive at these representations, one merely replaces $f^\alpha(t,x,v)$ by $g^\alpha(t, x-vt, v)$ within these quantities, institutes the change of variables $y = x-vt$ in the $v$ integrals, and for the field, further implements a change of variables in the convolution kernel \cite{Ionescu, Pankavich2022} of the form $\xi = \frac{z-x}{t}$.
Roughly speaking, in \cite{Bigorgne-Ruiz,Schlue-Taylor} the authors show that for large times, the expressions \eqref{rho} and \eqref{E} allow one to  expand
	\begin{equation}
	\label{eq:expansion}
	t^3 \rho(t,x) \approx \rho_0(t,x) + t^{-1} \rho_1(t,x) + \cdots + t^{-k} \rho_k(t,x)  + \cdots
	\end{equation}
and
	\begin{equation}
	\label{eq:expansion2}
	t^2 E(t,x) \approx E_0(t,x) + t^{-1} E_1(t,x) + \cdots + t^{-k} E_k(t,x)  + \cdots
	\end{equation}
where the functions $\rho_k$ and $E_k$  are uniformly bounded and possess limits as $t\to\infty$.
Here, we will identify the structure of these limits and determine conditions which ensure that each term within the expansion vanishes up to a chosen order.
In particular, these expressions will result by implementing within \eqref{rho} and \eqref{E} the following multi-dimensional Taylor expansion of $g^\alpha$ of order $\ell\in\N$:
	\begin{equation}
	\label{g:taylor}
	g^\alpha \left(t, y, \frac{x-y}{t}  \right)
	=
	\sum_{m=0}^{\ell} \sum_{|\beta| = m}  \frac{(-y)^\beta}{t^{m}\beta!} D^\beta_v\, g^\alpha \!\left(t, y, \frac{x}{t}  \right) + \sum_{|\beta| = \ell+1}  \frac{(-y)^\beta}{t^{\ell+1}\beta!}  D^\beta_v\, g^\alpha \!\left(t, y, \frac{x-\theta^\alpha y}{t}  \right) 
	\end{equation}
for some $\theta^\alpha\in[0,1]$.

\subsection{Notation}
We  define $\N_0=\{0,1,2,\dots\}$, whereas $\N=\{1,2,\dots\}$. As we have seen in Theorem \ref{oldT1}, the quantity $F^{\alpha,0}(t,v)=\int f^\alpha(t,x,v)\ dx$ plays a crucial role, as does its limit as $t\to\infty$, namely $F^{\alpha,0}_\infty(v)$. In order to perform and rigorously justify an expansion as in \eqref{eq:expansion} and \eqref{eq:expansion2}, we shall require spatial moments of velocity derivatives of $g^\alpha$ (or, equivalently, derivatives of $f^\alpha$) to be well defined. Let $K\in\N$ be given and assume that $f^\alpha\in C^{K}$. For any $\alpha = 1,\dots,N$ and $\ell = 0,\dots,K$ we let
	\begin{equation}
	\label{eq:f-al}
	F^{\alpha, \ell}(t, v) := \sum_{|\beta|=\ell} \frac{1}{\beta !} \int (-y)^\beta D_v^\beta g^\alpha(t,y,v) \ dy,
	\end{equation}
where $\beta=(\beta_1,\beta_2,\beta_3)$ with $\beta_i\in\N_0$ and $|\beta|=\beta_1+\beta_2+\beta_3$. 
These quantities generalize the spatial averages $F^{\alpha, 0}(t, v)$ via insertion of \eqref{g:taylor} up to order $\ell$ and upon removing powers of $t$.
As usual, we have $D_v^\beta=\partial_{v_1}^{\beta_1}\partial_{v_2}^{\beta_2}\partial_{v_3}^{\beta_3}$ as well as $(-y)^\beta=(-y_1)^{\beta_1}(-y_2)^{\beta_2}(-y_3)^{\beta_3}$, and $\beta!=\beta_1!\beta_2!\beta_3!$. With this, we define
	\begin{equation}
	\label{eq:rho-l}
	\rho_\ell(t,x) :=  \sum_{\alpha=1}^N q_\alpha F^{\alpha, \ell} \left( t, \frac{x}{t} \right).
	\end{equation}
For each term, $\rho_\ell$, in the expansion of the particle density, we can define the associated term in the expansion of the field by
$$ E_\ell(t,x) :=  \nabla_x (\Delta_x)^{-1} \rho_\ell(t,x) =  \frac{1}{4\pi} \sum_{\alpha=1}^N q_\alpha \int \frac{\xi}{|\xi|^3}  F^{\alpha, \ell} \left( t, \frac{x}{t} - \xi \right) d\xi.$$
It will be shown that $F^{\alpha, \ell}(t, v), \rho_\ell(t,x)$, and $E_\ell(t,x)$ converge, respectively, to
	\begin{equation}
	\label{eq:F-al-inf}
	F^{\alpha, \ell}_\infty(v)
	:=
	 \sum_{|\beta|=\ell} \frac{1}{\beta !} \int (-y)^\beta D_v^\beta f^\alpha_\infty(y,v) \ dy,	
	 \end{equation}
	\begin{equation}
	\label{eq:rho-l-inf}
	\rho_{\ell,\infty}(v)
	:=
	\sum_{\alpha=1}^N q_\alpha F^{\alpha, \ell}_\infty  ( v ),
	\end{equation}
and
	\begin{equation}
	\label{eq:E-l-inf}
	E_{\ell,\infty}(v) :=\nabla_v \left(\Delta_v \right)^{-1} \rho_{\ell,\infty}(v),
	\end{equation}
with the latter two limits occurring along $v = \frac{x}{t}$ and where $f_\infty^\alpha$ are the limiting functions introduced in Theorems \ref{oldT1} and \ref{oldT2}.
In order to write an expansion of the form \eqref{eq:expansion} precisely, we introduce notation for the bounds that we will obtain for remainder terms, which involve derivatives of the translated distribution functions; 
namely, for every $k \in \mathbb{N}$ we define
	\begin{equation}
	\label{eq:G^k_v}
	\mcG^k_v(t) := 1 + \max_{\alpha = 1, ..., N} \sum_{|\beta|=k} \| D_v^\beta g^\alpha(t) \|_\infty
	\end{equation}
and
	\begin{equation}
	\label{eq:G^k_x,v}
	\mcG^k_{x,v}(t) :=  1+ \max_{\alpha = 1, ..., N} \sum_{\substack{|\beta| + |\gamma|=k \\ |\gamma| > 0}} \| D_v^\beta D_x^\gamma  g^\alpha(t) \|_\infty,
	\end{equation}
whenever these derivatives exist. In particular, we note that in \eqref{eq:G^k_x,v} the quantity $\mcG^k_{x,v}(t)$ includes $k$th order $x$-derivatives, namely 
$\displaystyle  \max_{\alpha = 1, ..., N} \sum_{|\gamma|=k} \| D_x^\gamma g^\alpha(t) \|_\infty.$

\subsection{Main Results}
We are now ready to present our main results, which extend the asymptotic behavior described by Theorems \ref{oldT1} and \ref{oldT2} to higher order expansions and close the estimates with bounds on higher order velocity derivatives of $g^\alpha$.
\begin{theorem}
\label{T1}
Let $m \in \N_0$ be given. There exist nontrivial $f^\alpha \in C^{m+1} \left((0,\infty) \times \mathbb{R}^6 \right)$  for every $\alpha = 1, ..., N$ satisfying \eqref{VP} such that \eqref{Assumption} holds,  $\mcM = 0$, and
$$
\begin{gathered}
\| \rho(t)\|_\infty \sim t^{-m-3},\\
\| E(t)\|_\infty \sim t^{-m-2}.
\end{gathered}
$$
If $m = 0$, then we have 
$$\| \nabla_x E(t)\|_\infty \sim t^{-3}$$
and for every $\alpha = 1, ..., N$ there is $f^\alpha_\infty \in C_c^{m}(\bfR^6)$ such that
$$\sup_{(x,v) \in \bfR^6} \left | f^\alpha \left(t,x +vt - \frac{q_\alpha}{m_\alpha} \ln(t)E_{0,\infty}(v), v \right) - f^\alpha_\infty(x,v) \right |  \lesssim t^{-1}\ln^{4}(t).$$
In contrast, for $m \geq 1$  we have for every $k =1, ..., m$
$$\| \nabla_x^k E(t)\|_\infty \sim t^{-m-3}, $$
and for every $\alpha = 1, ..., N$ there is $f^\alpha_\infty \in C_c^{m}(\bfR^6)$ such that
$$\sup_{(x,v) \in \bfR^6} \left | f^\alpha(t,x +vt, v) - f^\alpha_\infty(x,v) \right |  \lesssim t^{-m}. $$
\end{theorem}

\begin{remark}
To the best of our knowledge, Theorem \ref{T1} is the first result to demonstrate decay rates of the charge density and electric field that are nontrivial (i.e., $\rho, E \not\equiv 0$) and faster than the dispersive rates (of $t^{-3}$ and $t^{-2}$, respectively) for \eqref{VP} in $\bfR^3$.
\end{remark}

Theorem \ref{T1} will be implied by the following result (taking $m = n+1$ and $\rho_{n+1,\infty} \not\equiv 0$) when $m \geq 1$, which, in turn, will be established inductively.

\begin{theorem}
\label{T2}
Let $n\in\N_0$ be given. Consider initial data $f_0^\alpha\in C_c^{n+2}(\R^6)$ launching solutions $f^\alpha\in C^{n+2}((0,\infty)\times\R^6)$ of \eqref{VP} for $\alpha = 1, ..., N$. If
\begin{equation*}
\rho_{\ell,\infty} \equiv 0 \qquad \mathrm{ for \ every \ } \ell \in \{0, ..., n\},
\end{equation*}
then for every $\alpha = 1, ..., N$ there is $f^\alpha_\infty \in C_c^{n+1}(\bfR^6)$ with $F^{\alpha, n+1}_\infty$, $\rho_{n+1,\infty}$ and $E_{n+1,\infty}$ defined by \eqref{eq:F-al-inf}, \eqref{eq:rho-l-inf}, and \eqref{eq:E-l-inf}, respectively,
such that
$$
\begin{gathered}
\sup_{x \in \mathbb{R}^3} \left | t^{n+4}\rho(t,x) - \rho_{n+1, \infty} \left( \frac{x}{t}\right)  \right | \lesssim t^{-1}, \\
\sup_{x \in \mathbb{R}^3} \left | t^{n+3} E(t,x) - E_{n+1,\infty} \left( \frac{x}{t}\right)  \right | \lesssim t^{-1}, \\ 
\end{gathered}
$$
with
%\begin{equation}
\begin{gather}
\sup_{x \in \mathbb{R}^3} \left | t^{n+3}\nabla^k_xE(t,x) - \nabla_v^k E_{n+1-k,\infty} \left( \frac{x}{t}\right)  \right | \lesssim t^{-1}, 
\label{eq:derivs-e} \\
\mcG^k_{x,v}(t) + \mcG^k_{v}(t) \lesssim 1\nonumber
\end{gather}
%\end{equation}
for every $k \in \{1, ..., n+1\}$,
and
$$
\begin{gathered}
\sup_{x \in \mathbb{R}^3} \left | t^{n+4}\nabla^{n+2}_xE(t,x) - \nabla_v^{n+2}E_{0,\infty} \left( \frac{x}{t}\right)  \right | \lesssim t^{-1} \ln(t), \\
\mcG^{n+2}_{x,v}(t) + \mcG^{n+2}_{v}(t) \lesssim 1.
\end{gathered}
$$
In each of these cases, we further have
$$\sup_{(x,v) \in \bfR^6} \left | f^\alpha(t,x +vt, v) - f^\alpha_\infty(x,v) \right |  \lesssim t^{-n-1}$$
and
$$\sup_{(x,v) \in \bfR^6} \left | \nabla_x^i \nabla_v^j f^\alpha \left(t,x+vt, v \right) - \nabla_x^i \nabla_v^j  f^\alpha_\infty(x,v) \right | \lesssim \max \left \{ t^{-n-1}, t^{j -n -2} \right \}$$
for $i,j \in \mathbb{N}_0$ with $1 \leq i+j \leq n+1$.
\end{theorem}

\begin{remark}
As mentioned above, Theorem \ref{T2} will be proved by induction. We note that we may lose decay on the spatial derivatives of $E(t,x)$ at every step of the induction argument. For instance, within the base case ($n=0$) and focusing on $k=1$, one could attempt to replace \eqref{eq:derivs-e} with the sharper estimate
$$\sup_{x \in \mathbb{R}^3} \left | t^4\nabla_x E(t,x) - \nabla_v E_{1,\infty}  \left( \frac{x}{t}\right)  \right | \lesssim t^{-1}.$$
In fact, because  $E_{0,\infty} \equiv 0$ in such a case, one would expect convergence of the field and its derivatives to the next order limit, namely $E_{1,\infty}$.
However, establishing a convergence result like this would require estimates of higher-order derivatives of $g^\alpha$, e.g., $\mcG_v^3(t) \lesssim 1$.
More generally, at the $n$th step of the induction argument we estimate $\mcG_v^{n+2}(t)$ in order to obtain a sharp rate of decay for $\rho(t,x)$ and $E(t,x)$. However, in order to achieve the optimal rate of decay for $k$th-order derivatives of these quantities, one would require estimates of $\mcG_v^{k+n+2}(t)$, which would further necessitate an estimate of $\Vert \nabla_x^{k+n+2} E(t) \Vert_\infty$. Thus, sharp estimates of derivatives require even greater regularity and lead to a loss of closure in the argument.
\end{remark}
\begin{remark}
\label{Finf}
Throughout the proof of Theorem \ref{T2}, we will establish and utilize the functions $F^{\alpha, \ell}_\infty(v)$, which will be initially understood only as the large time asymptotic limits of the functions $F^{\alpha, \ell}(t,v)$.
However, at the end of the proof we will justify the regularity needed to assert the ultimate relationship between the limiting distributions $f_\infty(x,v)$ and the limiting spatial moments $F^{\alpha, \ell}_\infty(v)$ given by \eqref{eq:F-al-inf}.
Hence, this relationship is not needed to establish our results. 
Only the existence of the limits $F^{\alpha, \ell}_\infty(v) = \lim_{t\to\infty} F^{\alpha, \ell}(t,v)$ will be necessary, independent of an explicit formula for the limiting functions. 
\end{remark}

\begin{remark}
It appears that the main physical mechanism that generates further decay of these quantities is \emph{charge cancellation} within the density $\rho(t,x)$, rather than dispersion or phase mixing, the latter of which occurs in the case of Landau damping \cite{BMM,Ionescu2, VM}. If we think of a plasma comprised of two identical, yet oppositely charged, species, this means that the decay comes from the distribution of the two species overlapping with one another in physical space asymptotically as $t \to \infty$. The precise rate of decay, then, has to do with how quickly they overlap.
To put this into a more concrete context, one may consider a two-species ($N=2$) plasma with $q_\pm = \pm1$, limiting Gaussian distributions $f_\infty^\pm(x,v) = \frac{1}{(2\pi \sigma_\pm^3)^{3/2}} \exp\left ({-\frac{|x-\mu_\pm|^2}{2\sigma_\pm^2}} \right ) \phi(v)$ for some smooth $\phi \not\equiv 0$ with $\int \phi(v) \ dv = 0$, $\mu_\pm \in \bfR$, and $\sigma_\pm > 0$. In this case the following possibilities arise:
\begin{itemize}
\item
If $\mu_+ \neq \mu_-$ and $\sigma_+ \neq \sigma_-$, then the distributions don't overlap much at all, though $\rho_{0,\infty} \equiv 0$ from \eqref{eq:F-al-inf} and \eqref{eq:rho-l-inf} due to the normalization of total positively/negatively charged particles.
\item
Instead, if $\mu_+ = \mu_-$ and $\sigma_+ \neq \sigma_-$, then these distributions share the same mean position, and hence their first moments cancel, yielding $\rho_{0,\infty} \equiv\rho_{1,\infty} \equiv 0$, but $\rho_{2,\infty} \not\equiv 0$.
\item
Finally, if $\mu_+ = \mu_-$ and $\sigma_+ = \sigma_-$, then these distributions overlap identically, which yields $\rho_{k,\infty} \equiv 0$ for all $k \in \mathbb{N}_0$ and $\rho(t,x) \equiv 0$ for all $t \geq 0$, $x \in \bfR^3$.
\end{itemize}
Such cancellation can be extended to an arbitrary number of moments, rather than merely the first or second, and this is implemented within the proof of Theorem \ref{T1} (Section \ref{sec:proofs}).
Moreover, this idea is reminiscent of the behavior of electric dipoles or multipoles \cite{Jackson}, in which charges of opposite sign located an infinitesimal distance apart lead to an atypical dominant term within the multipole expansion of the electric field, thereby generating field behavior that possesses faster spatial decay than generic solutions of Poisson's equation (e.g., $|E(t,x)| \sim |x|^{-3}$ for $|x|$ sufficiently large). Of course, if particle positions generally disperse at linear rates in time, this increased spatial decay transforms into time decay (e.g., $\Vert E(t)\Vert_\infty \sim t^{-3}$ for $t$ sufficiently large).

\end{remark}

\begin{remark}
The assumption \eqref{Assumption} was first used in \cite[Lemma 2.2]{Pankavich2022}, where it was shown that it implies the dispersive rates of $\|E(t)\|_\infty\lesssim t^{-2}$ and  $\|\rho(t)\|_\infty\lesssim t^{-3}$. This was established via an estimate on the spread of the spatial characteristics. The exponent $p = 5/3$ appearing in \eqref{Assumption} is the result of a simple algebraic constraint, and while it is conceivable that it could be relaxed, this is not the focus of the present article.
\end{remark}

\begin{remark}
Our methods are not constrained to dimension $d=3$ and could be used to arrive at similar decay results concerning small data solutions of \eqref{VP} for $d \geq 4$ \cite{PankavichHighd}.
\end{remark}

\begin{remark}
We believe that by combining the tools of \cite{BP} with the methods developed herein and the scattering map of \cite{Bigorgne2}, Theorems \ref{T1} and \ref{T2} can be extended to achieve analogous decay rates and self-similar expansions for small data solutions of the relativistic Vlasov-Maxwell system (see also \cite{Bigorgne1}).
\end{remark}

\subsection{Strategy of the Proof}
To prove the theorems we reformulate the original problem within a dispersive reference frame that is co-moving with the particles.
More specifically, defining $g^\alpha$ as in \eqref{eq:g} and applying the aforementioned change of variables (see \eqref{rho} and \eqref{E}) to the field and charge density, \eqref{VP} becomes
\begin{equation}
\tag{VP$_g$}
\label{VPg}
\left \{ \begin{aligned}
& \partial_{t}g^\alpha - \frac{q_\alpha}{m_\alpha} t E(t,x+vt)\cdot \nabla_{x}g^\alpha+ \frac{q_\alpha}{m_\alpha} E(t,x+vt) \cdot\nabla_{v}g^\alpha=0, \qquad \alpha = 1, ..., N\\
& E(t,x) = \frac{1}{4\pi t^2} \sum_{\alpha=1}^N  q_\alpha \iint \frac{\xi}{|\xi|^3} \ g^\alpha \left(t, y, \frac{x-y}{t} - \xi \right) \ dy d\xi
\end{aligned} \right .
\end{equation}
with
$$\rho(t,x)= t^{-3} \sum_{\alpha=1}^N q_\alpha \int_{\mathbb{R}^3} g^\alpha \left(t,y, \frac{x-y}{t} \right)\,dy$$
and the initial conditions $g^\alpha(0,x,v) = f^\alpha_0(x,v)$. 
As shown within \cite{Pankavich2022}, the spatial support and velocity derivatives of $g^\alpha$ can grow logarithmically in time or become uniformly bounded, depending upon the rate of field decay (see Lemmas \ref{Sg}, \ref{Dng0}, and \ref{Dng}), while the corresponding spatial support and $v$-derivatives of $f^\alpha$ grow at least linearly in time.
Moreover, the Taylor expansion in \eqref{g:taylor} and the dependence of the charge density on spatial integrals of $g^\alpha$ within \eqref{eq:f-al} and \eqref{eq:rho-l} inherently demonstrate the need to control the growth of the spatial support and velocity derivatives of these translated distribution functions in order to gain an extra power of time decay within each term of the expansion.
Finally, capturing the exact asymptotic behavior, rather than mere decay estimates, at each step of the induction argument within the proof of Theorem \ref{T2} enables us to establish the next order limiting distribution of $F^{\alpha,\ell}(t,v)$ and propagate the convergence of the charge density and electric field to their precise limits from one step to the next.

\section{Preliminary Lemmas}
\label{Lemmas}
This section is dedicated to a sequence of lemmas concerned with estimates of the field, the charge density, and their respective derivatives. As these lemmas will be used to prove Theorem \ref{T2}, we will assume that $\mcM = 0$ and $\rho_{0,\infty} \equiv 0$ throughout this section. Hence, the result of Theorem \ref{oldT2} applies immediately.
%To help the reader, let us summarize the results of the lemmas appearing in this section:
%	\begin{itemize}
%	\item
%	Lemma \ref{LDE}: bound on  $\Vert \nabla^k_xE(t) \Vert_\infty$.
%	\item
%	Lemma \ref{dvkrho}: bound on  $\infnorm{\nabla^k_x \rho(t)}$.
%	\item
%	Lemma \ref{dvkFalpha0}: bound on  $\infnorm{\nabla_v^{k} F^{\alpha,0}(t) - \nabla_v^{k} F^{\alpha,0}_\infty} $.
%	\item
%	Lemma \ref{dvkFell}: bound on  $\infnorm{\nabla_v^k F^{\alpha,\ell}(t)-\nabla_v^k F^{\alpha,\ell}_\infty}$.
%	\item
%	Lemma \ref{DEgeneral}: bound on  $\sup_{x \in \bfR^3} \left | t^{k+\ell+2} \nabla^k_xE(t,x) - \nabla^k_vE_{\ell,\infty}\left(\frac{x}{t} \right) \right |$ by $k$ derivatives of $F^{\alpha,\ell}$.
%	\item
%	Lemma \ref{Drhogeneral}: bound on  $\sup_{x \in \bfR^3} \left | t^k \nabla^k_x\rho_\ell(t,x) - \nabla^k_v\rho_{\ell,\infty}\left(\frac{x}{t} \right) \right |$.
%	\item
%	Lemma \ref{LDEgeneral}: bound on  $\sup_{x \in \bfR^3} \left | t^{k+\ell+2} \nabla^k_xE(t,x) - \nabla^k_vE_{\ell,\infty}\left(\frac{x}{t} \right) \right | $  by $k-1$ derivatives of $F^{\alpha,\ell}$ (and a $\ln(t)$).
%	\item
%	Lemma \ref{LDensitygeneral}: bound on  $\sup_{x \in \bfR^3} \left | t^{\ell+3} \rho(t,x) - \rho_{\ell,\infty} \left(\frac{x}{t} \right) \right |$.
%%	\item
%%	Lemma \ref{LFieldgeneral}: bound on  $\sup_{x \in \bfR^3} \left | t^{\ell+2} E(t,x) - E_{\ell,\infty}\left(\frac{x}{t} \right) \right |$.
%	\end{itemize}

\subsection{Further Notation and Preliminary Estimates}
Prior to stating the lemmas, we first introduce some notation related to the translated distribution functions.
As mentioned previously, we let
$$g^\alpha(t,x,v) = f^\alpha(t,x+vt, v)$$
%and because the translation alters the spatial characteristics of this system, we further define
%$$\mcY(t,\tau, x,v) = \mcX(t,\tau, x, v) - t \mcV(t, \tau, x, v)$$
%so that
%$$ \dot{\mcY}(t) = - \frac{q_\alpha}{m_\alpha} t E(t, \mcY(t) + t \mcV(t))$$
%with $\mcY(\tau) = x - v\tau$. 
%In addition,
and note that $\Vert g^\alpha(t) \Vert_\infty \leq \Vert f_0^\alpha \Vert_\infty$ for all $t \geq 0$ due to \eqref{VPg}.
As our approach relies heavily upon the growth of the  support of $g^\alpha$, 
we  define the sets
	\[
	\forall t\geq0,\qquad\bS(t) = \overline{\left \{ (x,v) \in \R^6 : g^\alpha(t,x,v) \neq 0 \right \}},
	\]
and
	\begin{equation*}
	\label{eq:supp-x-g}
	\forall t\geq0,\forall v\in\R^3,\qquad\bS_x(t,v) = \overline{\left \{x \in \bfR^3 : g^\alpha(t,x,v) \neq 0 \right \}}
	\end{equation*}
%and
%	\begin{equation*}
%	\forall t\geq0,\forall x\in\R^3,\qquad\bS_v(t,x) = \overline{\left \{v \in \bfR^3 : g^\alpha(t,x,v) \neq 0 \right \}}
%	\end{equation*}
and the quantity
$$\mu(t) = \max_{\alpha = 1, ..., N} \sup_{v\in \bfR^3} \left | \bS_x(t,v) \right |.$$

We also introduce some notation related to multi-indices, as these can become cumbersome.
In the forthcoming discussion, the Greek letters $\beta$ and $\gamma$ shall denote multi-indices of length $3$, with a subscript clarifying whether they belong to derivatives acting on the $x$ variable or the $v$ variable. They are always elements of $\N_0^3$, i.e. their entries can take the values $\{0,1,2,\dots\}$. Hence, for $\beta_v=(\beta_{v,1},\beta_{v,2},\beta_{v,3})$ and $\beta_x=(\beta_{x,1},\beta_{x,2},\beta_{x,3})$  we write derivatives as follows:	\[
	D_v^{\beta_v}
	=
	\partial_{v_1}^{\beta_{v,1}}\partial_{v_2}^{\beta_{v,2}}\partial_{v_3}^{\beta_{v,3}}
\qquad \mathrm{and} \qquad
	D_x^{\beta_x}
	=
	\partial_{x_1}^{\beta_{x,1}}\partial_{x_2}^{\beta_{x,2}}\partial_{x_3}^{\beta_{x,3}}.
	\]
These can be concatenated, so that $D_v^{\beta_v}D_x^{\beta_x}$ is given by
	\[
	D_v^{\beta_v}D_x^{\beta_x}
	=
	\partial_{v_1}^{\beta_{v,1}}\partial_{v_2}^{\beta_{v,2}}\partial_{v_3}^{\beta_{v,3}}\partial_{x_1}^{\beta_{x,1}}\partial_{x_2}^{\beta_{x,2}}\partial_{x_3}^{\beta_{x,3}}.
	\]
In the case of a vector $E=(E_1,E_2,E_3)$, the derivatives are applied to each coordinate separately and the result remains a vector, i.e.
	\[
	D_x^{\beta_x}E
	=
	\left(D_x^{\beta_x}E_1,D_x^{\beta_x}E_2,D_x^{\beta_x}E_3\right).
	\]
We can further define a relation $\preceq$ on multi-indices. More specifically, for $\beta,\gamma\in\N_0^3$ we write
	\[
	\gamma \preceq\beta\qquad\text{whenever}\qquad\gamma_i\leq\beta_i \quad \text{for all} \ i=1,2,3,
	\]
and, in this case, we define
	\[
	\beta-\gamma:=(\beta_1-\gamma_1,\beta_2-\gamma_2,\beta_3-\gamma_3),
	\]
which is also an element of $\N_0^3$. Lastly, the magnitude of a multi-index is the sum of its elements; namely, for $\beta=(\beta_1,\beta_2,\beta_3)$, we define
	\[
	|\beta|=\beta_1+\beta_2+\beta_3.
	\]
Note that for $\eta\preceq\gamma\preceq\beta$ we have
$$	|\beta-\gamma|+|\gamma-\eta|
	=
	\sum_{i=1}^3(\beta_i-\gamma_i)+\sum_{i=1}^3(\gamma_i-\eta_i)\\
	=
	\sum_{i=1}^3(\beta_i-\eta_i)\\
	=
	|\beta-\eta|.
$$

Occasionally, we use an integer superscript, such as $D_x^k$, where $k\in\N$. This symbol means that we are taking an \emph{arbitrary} $k$th order derivative: that is, for $j\in\{1,\dots,k\}$ and some choice $i_j\in\{1,2,3\}$,
	\[
	D^k_x=\partial_{x_{i_1}}\cdots \partial_{x_{i_k}}.
	\]
This holds similarly for an arbitrary $k$th order derivative $D^k_v$ in the $v$ variable.

%We note that when the superscript of a derivative term is a scalar rather than a multi-index, e.g. $D^k_x$ for $k \in \mathbb{N}$, then this term represents an arbitrary derivative of that order, e.g. any $k$th order derivative.

%

%With suitable decay of the field and its derivatives, we now study the behavior of the translated characteristics and establish a measure-preserving property as for the characteristics of \eqref{VP}.
%
%\begin{lemma}
%\end{lemma}
%
%\begin{proof}
%\end{proof}

Prior to diving into our sequence of lemmas, we begin by recalling some important properties of the translated distribution functions from \cite{Pankavich2022}.
First, we note that the measure of the phase space support of each $g^\alpha$ is preserved in time and recall an estimate on the spatial support of each $g^\alpha$ that will ultimately provide a uniform-in-time bound.
\begin{lemma}[\cite{Pankavich2022}]
\label{Sg}
For every $\alpha = 1, ..., N$ and $t \geq 0$, we have 
$$| \bS(t) | = | \bS(0) |.$$
Additionally, the measure of the maximal spatial support satisfies
$$\mu(t) \lesssim \left( 1 + \int_1^t s \| E(s) \|_\infty \ ds \right)^3.$$
\end{lemma}
\begin{proof}
The proof in \cite{Pankavich2022} follows from an analysis of the evolution of $g^\alpha$ in phase space; specifically, the divergence-free property of the Vlasov flow implies the first statement. The result for $\mu(t)$ simply follows from estimating the rate at which the translated spatial characteristics separate.
\end{proof}

Due to the assumption that $\rho_{0,\infty} \equiv 0$, we find 
$$\Vert E(t) \Vert_\infty \lesssim t^{-3}$$
from Theorem \ref{oldT2}.
Because the field is bounded on any bounded time interval, inserting this estimate within Lemma \ref{Sg} gives
\begin{equation}
\label{Sxg}
\mu(t) \lesssim 1.
\end{equation}
Hence, both the phase space support and spatial support of $g^\alpha(t)$ are uniformly bounded for every $\alpha = 1, .., N$.

Next, we state an important (and well-known) estimate for the electric field.
\begin{lemma}
\label{LE}
For any $\phi \in C_c(\bfR^3)$, there is $C > 0$ such that
$$\| \nabla (\Delta)^{-1} \phi \|_\infty \leq C \| \phi \|_\infty.$$
\end{lemma}

\begin{proof}
First we note that for $\phi \in C_c(\bfR^3)$, we have
$$ \| \phi \|_1 \leq \left |\mathrm{supp}(\phi) \right | \| \phi \|_\infty \leq C \| \phi \|_\infty$$
due to the compact support of $\phi$.
Finally, the proof is completed by using this within the classical estimate (cf. \cite{Hormander}, Lemma 4.5.4)
$$\| \nabla (\Delta)^{-1} \phi \|_\infty \leq C \| \phi \|_1^{1/3}  \| \phi \|_\infty^{2/3}$$
for any $\phi \in C_c(\bfR^3) \subset L^1(\bfR^3) \cap L^\infty(\bfR^3)$.
\end{proof}

With this, we note that, as the spatial support of $g^\alpha(t,x,v)$ is uniformly bounded in time by \eqref{Sxg}, the lemma implies
$$\| E(t) \|_\infty = \| \nabla_x (\Delta_x)^{-1} \rho(t) \|_\infty \lesssim\| \rho(t) \|_\infty,$$
with analogous results for field derivatives and any function dependent upon the spatial support of a particle distribution.

\subsection{Preliminary Bounds on the Particle Density and the Field}
The next lemma provides an initial estimate of $k$th order field derivatives in terms of the $k$th order derivatives of the distribution function.
In order to precisely capture this dependence, we first define the function 
$$ \ln^*(x) = \begin{cases}
\ln(x), & x > 1\\
0, & x \leq 1.
\end{cases}
$$

\begin{lemma}
\label{LDE}
For any $k \in \mathbb{N}$, if
$\mcG^{k-1}_v(t) \lesssim 1$,
then
$$ \Vert \nabla^k_xE(t) \Vert_\infty \lesssim t^{-2-k} \left(1 + \ln^* \left( \max_{\alpha = 1, ..., N} \Vert \nabla^k_v g^\alpha(t) \Vert_\infty \right) \right). $$
\end{lemma}
\begin{proof}
We begin by taking any $k$th-order derivative with respect to $x$ of the field representation \eqref{E} so that
$$ D^{k}_x E(t,x) = \frac{1}{4\pi t^{2+k}} \int \frac{\xi}{|\xi|^3} \sum_{\alpha = 1}^N q_\alpha \int D^{k}_v g^\alpha \left(t, y, \frac{x-y}{t} - \xi \right)  dy \, d\xi,$$
or written another way
$$  t^{2+k}D^{k}_x E(t,x) = \frac{1}{4\pi} \int \frac{\xi}{|\xi|^3} \partial_{z_i}\mcA \left(t,  \frac{x}{t} - \xi \right)  d\xi$$
for some $i \in \{1, 2, 3\}$ where 
$$\mcA(t,z) = \sum_{\alpha = 1}^N q_\alpha \int D^{k-1}_v g^\alpha \left(t, y, z -\frac{y}{t} \right)  dy$$
and the $k$th order derivative is decomposed into the composition of first and $(k\!-\!1)$st order derivatives via $D_v^k = \partial_{v_i} D_v^{k-1}$.
Because $g^\alpha(t,x,v)$ has bounded support in $x$ due to \eqref{Sxg} and in the phase space $(x,v)$ due to Lemma \ref{Sg}, we find
\begin{equation}
\label{H1}
\Vert \mcA(t) \Vert_1 + \Vert \mcA(t) \Vert_\infty \lesssim \max_{\alpha = 1, ..., N} \Vert D^{k-1}_v g^\alpha(t) \Vert_\infty \lesssim \mcG^{k-1}_v(t) \lesssim 1
\end{equation}
and
\begin{equation}
\label{H2}
\Vert \partial_{z_i}\mcA(t) \Vert_\infty \lesssim \max_{\alpha = 1, ..., N} \Vert \nabla^k_v g^\alpha(t) \Vert_\infty.
\end{equation}
Then, using the identity
$$\partial_{z_i}\mcA \left(t,  \frac{x}{t} - \xi \right) = - \partial_{\xi_i} \left [ \mcA \left(t,  \frac{x}{t} - \xi \right) \right ]$$ 
with an integration by parts away from the singularity and taking some $0<d<R<\infty$ to be chosen later, we find
\begin{eqnarray*}
t^{2+ k}  \left | D^{k}_x E^j(t,x) \right | & \leq & \left | \int_{|\xi| < d} \frac{\xi_j}{|\xi|^3} \partial_{z_i}\mcA \left(t, \frac{x}{t} - \xi \right) d\xi \right | + \left | \int_{|\xi| = d} \frac{\xi_j\xi_i}{|\xi|^4} \mcA \left(t, \frac{x}{t} - \xi \right) dS_\xi \right |\\
& &  + \left | \int_{d < |\xi| < R}  \partial_{\xi_i}  \left( \frac{\xi_j}{|\xi|^3} \right)  \mcA \left(t, \frac{x}{t} - \xi \right) d\xi \right |  + \left | \int_{|\xi| > R} \partial_{\xi_i}  \left( \frac{\xi_j}{|\xi|^3} \right)  \mcA \left(t, \frac{x}{t} - \xi \right) d\xi \right |\\
& =: &  I + II + III + IV.
\end{eqnarray*}
The estimate of $I$ merely uses \eqref{H2} so that
$$I \lesssim d \Vert \partial_{z_i}\mcA(t) \Vert_\infty \lesssim d \max_{\alpha = 1, ..., N} \Vert \nabla^k_v g^\alpha(t) \Vert_\infty .$$
We use \eqref{H1} to estimate $II$, $III$, and $IV$, which yields
$$II \lesssim \| \mcA(t) \|_\infty \left( \int_{|\xi| = d} |\xi|^{-2} dS_\xi \right) \lesssim 1,$$
$$III \lesssim \int_{d < |\xi| < R} |\xi|^{-3}  \mcA \left(t, \frac{x}{t} - \xi \right)  \ d\xi  \lesssim \ln \left(\frac{R}{d}\right)  \| \mcA(t)\|_\infty \lesssim \ln \left(\frac{R}{d}\right),$$
and
$$IV \lesssim \int_{|\xi| > R} |\xi|^{-3}  \mcA \left(t, \frac{x}{t} - \xi \right) d\xi  \lesssim R^{-3} \| \mcA(t)\|_1 \lesssim R^{-3},$$
respectively.
Collecting these estimates yields
$$t^{2+k} \Vert D^k_xE(t) \Vert_\infty \lesssim 1 + d \max_{\alpha = 1, ..., N}\Vert \nabla^k_v g^\alpha(t) \Vert_\infty + \ln \left(\frac{R}{d}\right) + R^{-3}.$$
We choose $R=2$ and
	\[
	d(t) = \min \left \{1, \left ( \max_{\alpha = 1, ..., N}\Vert \nabla^k_v g^\alpha (t) \Vert_\infty \right)^{-1} \right \}.
%\begin{cases}
%1,&\quad\text{if }\displaystyle \max_{\alpha = 1, ..., N} \Vert \nabla^k_v g^\alpha(t) \Vert_\infty \leq 1,\\
%\left( \displaystyle \max_{\alpha = 1, ..., N}\Vert \nabla^k_v g^\alpha (t) \Vert_\infty \right)^{-1}, &\quad\text{if }\displaystyle \max_{\alpha = 1, ..., N} \Vert \nabla^k_v g^\alpha(t) \Vert_\infty > 1.
%\end{cases}
	\]
Then for all $t \geq 0$, we have 
$$d(t) \cdot \left (\max_{\alpha = 1, ..., N}\Vert \nabla^k_v g^\alpha(t) \Vert_\infty \right ) \leq 1$$ 
and 
$$\ln \left(\frac{1}{d(t)}\right)=\ln^*\left(\max_{\alpha = 1, ..., N}\Vert \nabla^k_v g^\alpha (t) \Vert_\infty\right),$$ 
so that the previous estimate yields
%
%For $\displaystyle \max_{\alpha = 1, ..., N} \Vert \nabla^k_v g^\alpha(t) \Vert_\infty \leq 1$,  we choose $d = 1$ and $R = 2$, so that 
%$$t^{2+k} \Vert D^k_xE(t) \Vert_\infty \lesssim 1.$$
%Alternatively, for $\displaystyle \max_{\alpha = 1, ..., N}\Vert \nabla^k_v g^\alpha(t) \Vert_\infty > 1$, choosing $d = \left( \displaystyle \max_{\alpha = 1, ..., N}\Vert \nabla^k_v g^\alpha (t) \Vert_\infty \right)^{-1}$ with $R= 1$ yields
%$$t^{2+k} \Vert D^k_xE(t) \Vert_\infty \lesssim 1 + \ln\left( \max_{\alpha = 1, ..., N} \Vert \nabla^k_v g^\alpha (t) \Vert_\infty \right).$$
%Combining these cases gives
$$t^{2+k} \Vert D^k_xE(t) \Vert_\infty \lesssim 1 + \ln^*\left( \max_{\alpha = 1, ..., N} \Vert \nabla^k_v g^\alpha (t) \Vert_\infty \right).$$
Finally, we have obtained the same bound for every such derivative of the field, yielding the stated result.
\end{proof}

Next, we control derivatives of the charge density in terms of those of the distribution function.

\begin{lemma}
\label{dvkrho}
For every $k \in \mathbb{N}$, we have 
$$\infnorm{\nabla^k_x \rho(t)}\lesssim t^{-3-k} \mcG_v^k(t).$$
\end{lemma}
\begin{proof}
Taking any $k$th order derivative with respect to $x$ of $\rho(t,x)$ via the representation \eqref{rho} yields
$$D^k_x \rho(t,x+vt)=t^{-3-k}\sum_{\alpha=1}^N q_\alpha \int D^k_v g^\alpha\left(t,y,v+\frac{x-y}{t}\right)dy,$$
and, as the spatial support of $g^\alpha$ is uniformly bounded for each $\alpha = 1, ..., N$ by \eqref{Sxg}, we take supremums so that 
the gradient satisfies
$$\infnorm{\nabla^k_x \rho(t)}\lesssim t^{-3-k} \mcG_v^k(t).$$
\end{proof}

\subsection{Bounds on Derivatives of $F$}

In order to estimate field derivatives later, we will first need to determine the large time rate of convergence of higher-order derivatives of $F^{\alpha, 0}(t,v)$. 

\begin{lemma}
\label{dvkFalpha0}
	For every $\alpha = 1, ..., N$ and $k \in \N_0$, we have 
	\begin{equation*}
		\infnorm{\nabla_v^{k} F^{\alpha,0}(t) - \nabla_v^{k} F^{\alpha,0}_\infty} \lesssim \int_t^\infty  \sum_{j=0}^{k} s^{k-j}\biggl [s \infnorm{\nabla_x^{k-j} \rho(s)}\mcG_v^j(s) + \infnorm{\nabla_x^{k-j}E(s)}\mcG_v^{j+1}(s) \biggr ]ds.
	\end{equation*}
	under the assumption that the integral on the right side is finite.
\end{lemma}
\begin{proof}
	We first take the time derivative of $F^{\alpha,0}(t)$ and write $\partial_tg^\alpha$ using \eqref{VPg} to find
	\begin{align*}
		\partial_t F^{\alpha,0}(t,v)&= \int \partial_t g^\alpha(t,x,v) \ dx\\
		&=\int \frac{q_\alpha}{m_\alpha}\Big[ t E(t,x+vt)\cdot \nabla_x g^\alpha(t,x,v) - E(t,x+vt)\cdot\nabla_v g^\alpha(t,x,v) \Big]dx.
	\end{align*}
	Integrating by parts within the first term, the $x$-derivative is transferred from $g^\alpha$ to $E$ and simplified 
	using \[\nabla_x\cdot E(t,x)=\rho(t,x).\]
	The boundary term is zero due to the compact spatial support of $g^\alpha$ from \eqref{Sxg}, and this first term becomes
	\begin{align*}
		\int tE(t,x+vt)\cdot \nabla_x g^\alpha(t,x,v) \, dx=-\int t\rho(t,x+vt)g^{\alpha}(t,x,v) \, dx.
	\end{align*}
	Next, we take any $k$th order velocity derivative, denoted by $D^\beta_v$ with $|\beta| = k$, of the original equality, yielding 
	\begin{align*} 
		&\partial_t D_v^\beta F^{\alpha,0}(t,v) = -\frac{q_\alpha}{m_\alpha} \int \Big[ t D_v^\beta \big(\rho(t,x+vt)g^{\alpha}(t,x,v) \big) + D_v^\beta \big( E(t,x+vt)\cdot\nabla_v g^\alpha(t,x,v)\big) \Big] dx\\
		&\qquad =-\frac{q_\alpha}{m_\alpha} \int \sum_{j=0}^{k}\binom{k}{j} \sum_{\substack{|\gamma|=j\\ \gamma\preceq\beta}} \Big[ t D_v^{\beta-\gamma} \rho(t,x+vt)D_v^\gamma g^{\alpha}(t,x,v) + D_v^{\beta-\gamma}E(t,x+vt)\cdot D_v^\gamma \nabla_v g^\alpha(t,x,v) \Big] dx\\
		&\qquad =-\frac{q_\alpha}{m_\alpha} \int \sum_{j=0}^{k}\binom{k}{j} \sum_{\substack{|\gamma|=j\\ \gamma\preceq\beta}} \Big[ t^{1+|\beta-\gamma|} D_x^{\beta-\gamma} \rho(t,x+vt)D_v^\gamma g^{\alpha}(t,x,v) \\
		&\qquad \hspace{5cm}+ t^{|\beta-\gamma|}D_x^{\beta-\gamma}E(t,x+vt)\cdot D_v^\gamma \nabla_v g^\alpha(t,x,v) \Big] dx.
	\end{align*} 
	Taking the supremum over $x,v\in\R^3$ and using $|\beta-\gamma|=k-j$, we find
	\begin{align*}
		\infnorm{\partial_t D_v^\beta F^{\alpha,0}(t)} %&\leq \left|\frac{q_\alpha}{m_\alpha}\right| \sum_{|\beta|=k}\sum_{j=0}^{k}\binom{k}{j} \sum_{\substack{|\gamma|=j\\ \gamma\preceq\beta}} \Big[ t^{1+|\beta-\gamma|} \infnorm{D_x^{|\beta-\gamma|} \rho(t)}\mcG_v^j(t) + t^{|\beta-\gamma|}\infnorm{D_x^{|\beta-\gamma|}E(t)}\mcG_v^{j+1}(t)\Big]\\
		&\lesssim \sum_{j=0}^k \biggl [ t^{1+k-j} \infnorm{\nabla_x^{k-j} \rho(t)}\mcG_v^j(t) + t^{k-j}\infnorm{\nabla_x^{k-j}E(t)}\mcG_v^{j+1}(t) \biggr ].
	\end{align*} 
	Hence, for $\tau \geq t$ we can write
	\[\infnorm{D_v^{\beta} F^{\alpha,0}(t) -D_v^{\beta} F^{\alpha,0}(\tau)} = \infnorm{\int_t^\tau \partial_t D_v^{\beta} F^{\alpha,0}(s) \ ds }\leq \int_t^\tau \infnorm{\partial_t D_v^{\beta} F^{\alpha,0}(s)} ds . \]
	Then, as long as the integrand decays sufficiently fast as $s \to \infty$, we let $\tau \rightarrow \infty$ and find
	$$ \infnorm{D_v^{\beta} F^{\alpha,0}(t) - D_v^{\beta} F^{\alpha,0}_\infty} \lesssim \int_t^\infty  \sum_{j=0}^{k} s^{k-j}\biggl [s \infnorm{\nabla_x^{k-j} \rho(s)}\mcG_v^j(s) + \infnorm{\nabla_x^{k-j}E(s)}\mcG_v^{j+1}(s) \biggr ] ds.$$
Summing over all such derivatives with $|\beta| = k$ gives the desired result for any $\alpha= 1,\hdots,N$.
	%provided the integrals converge.
\end{proof}

\begin{lemma}
\label{dvkFell}
For every $\alpha = 1, \hdots, N$, $k \in \N_0$, and $\ell\in\mathbb{N}$, the spatial density $\nabla_v^k F^{\alpha,\ell}(t,v)$ satisfies
$$\infnorm{\nabla_v^k F^{\alpha,\ell}(t)-\nabla_v^k F^{\alpha,\ell}_\infty} \lesssim \int_t^\infty \sum_{j=0}^{k+\ell} s^{k+\ell-j}\infnorm{\nabla_x^{k+\ell-j}E(s)} \Big( s \mcG_{x,v}^{j+1}(s) + \mcG_v^{j+1}(s) \Big) \, ds$$
under the assumption that the integral on the right side is finite.
\end{lemma}
\begin{proof}
    Taking a time derivative of $F^{\alpha,\ell}(t,v)$ and using \eqref{VPg}, we find
    \begin{align*}
        \partial_t F^{\alpha,\ell}(t,v)&=\sum_{|\beta|=\ell}\int {(-x)}^\beta D_v^\beta \partial_t g^\alpha(t,x,v) \ dx\\
        &=\frac{q_\alpha}{m_\alpha}\sum_{|\beta|=\ell}\int {(-x)}^\beta D_v^\beta  \sqrP{ E(t,x+vt)\cdot \Big( t\nabla_x g^\alpha(t,x,v) - \nabla_v g^\alpha(t,x,v) \Big) }dx.
    \end{align*}
	Next, we apply any $v$-derivative of order $k$ to this equation yielding
	\begin{align*}
        \partial_t D_v^k F^{\alpha,\ell}(t,v) &= \frac{q_\alpha}{m_\alpha}\sum_{|\beta|=\ell} \int {(-x)}^\beta D_v^\beta D_v^k \sqrP{ E(t,x+vt)\cdot \Big( t\nabla_x g^\alpha(t,x,v) - \nabla_v g^\alpha(t,x,v) \Big) }dx.
    \end{align*}
    Applying these derivatives to the product and using the bounded spatial support of $g^\alpha$ from \eqref{Sxg}, we find 
%    \begin{align*}
%        \infnorm{\partial_t D_v^k F^{\alpha,\ell}(t)} &\lesssim \left|\frac{q_\alpha}{m_\alpha}\right|
%        \sum_{j=0}^{k+\ell}\binom{k\!+\!\ell}{j} t^{k+\ell-j}\infnorm{D_x^{k+\ell-j}E(t)} \Big[ t \mcG_{x,v}^{j+1}(t) + \mcG_v^{j+1}(t) \Big] 
%    \end{align*}
%    and thus
    \begin{align*}
        \infnorm{\partial_t D_v^k F^{\alpha,\ell}(t)} &\lesssim \sum_{j=0}^{k+\ell} t^{k+\ell-j}\infnorm{\nabla_x^{k+\ell-j}E(t)} \Big[ t \mcG_{x,v}^{j+1}(t) + \mcG_v^{j+1}(t) \Big]
    \end{align*}
    where $j = 0, ..., k+\ell$ represents the number of derivatives applied to the distribution function within the product.
    Taking $\tau\geq t$, we can write
    \[\infnorm{D_v^k F^{\alpha,\ell}(t)-D_v^k F^{\alpha,\ell}(\tau)} 
	=\infnorm{\int_t^\tau \partial_t D_v^k F^{\alpha,\ell}(s) \ ds} 
	\leq \int_t^\tau \infnorm{\partial_t D_v^k F^{\alpha,\ell}(s)} \ ds. \]
    Letting $\tau \to \infty$, we have 
	\begin{align*}
		\infnorm{D_v^k F^{\alpha,\ell}(t)-D_v^k F^{\alpha,\ell}_\infty} \lesssim \int_t^\infty 
        \sum_{j=0}^{k+\ell} s^{k+\ell-j}\infnorm{\nabla_x^{k+\ell-j}E(s)} \Big( s \mcG_{x,v}^{j+1}(s) + \mcG_v^{j+1}(s) \Big) \, ds,
	\end{align*}
	and summing over all such derivatives provides the desired result.
\end{proof}

\subsection{Improved Bounds on Gradients of the Particle Density and the Field}

The following lemma allows us to estimate lower-order field derivatives.

\begin{lemma}
\label{DEgeneral}
For any $k \in \N_0$ and $\ell \in \mathbb{N}$, we have 
\begin{align*}
	\sup_{x \in \bfR^3} \left | t^{k+\ell+2} \nabla^k_xE(t,x) - \nabla^k_vE_{\ell,\infty}\left(\frac{x}{t} \right) \right | &\lesssim \sum_{m=0}^{\ell-1} t^{\ell -m} \| \nabla_x^k \rho_m(t) \|_\infty + \max_{\alpha = 1, ..., N} \| \nabla^k_v F^{\alpha,\ell}(t)  - \nabla^k_v F^{\alpha,\ell}_\infty \|_\infty \\
	&\qquad + t^{-1} \mcG_v^{k+\ell+1}(t).
\end{align*}
In the case $\ell = 0$, this reduces to
$$\sup_{x \in \bfR^3} \left | t^{k+2} \nabla^k_xE(t,x) - \nabla^k_vE_{0,\infty}\left(\frac{x}{t} \right) \right | \lesssim \max_{\alpha = 1, ..., N} \| \nabla^k_v F^{\alpha,0}(t)  - \nabla^k_v F^{\alpha,0}_\infty \|_\infty + t^{-1} \mcG_v^{k+1}(t)$$
for any $k \in \N_0$.
\end{lemma}
\begin{proof}
We begin by taking any $k$th-order derivative with respect to $x$, denoted by $D_x^k$, of the $j$th component of \eqref{E}, subtracting the corresponding $k$th order $v$ derivative, denoted $D_v^k$, of the limiting field given by \eqref{eq:E-l-inf}, and using the Taylor series expansion \eqref{g:taylor} of $g^\alpha$, we find
\begin{align*}
t^{k+\ell+2} D^k_xE^j(t,x) &- D^k_v E^j_{\ell,\infty}\left(\frac{x}{t} \right) \\
& = \frac{1}{4\pi} \int \frac{\xi_j}{|\xi|^3}  \sum_{\alpha = 1}^N q_\alpha \left( t^\ell \int D^k_v \, g^\alpha \!\left(t, w, \frac{x-w}{t} - \xi \right)  dw - D^k_v F^{\alpha,\ell}_\infty \left(\frac{x}{t} - \xi \right) \right)\ d\xi\\
& = \frac{1}{4\pi} \biggl (\int \frac{\xi_j}{|\xi|^3}  \sum_{\alpha = 1}^N q_\alpha \biggl [ \, \sum_{m=0}^{\ell-1} \sum_{|\beta| = m} t^{\ell-m} \frac{1}{\beta!}\int (-w)^\beta D^k_vD^\beta_v\, g^\alpha \!\left(t, w, \frac{x}{t} - \xi \right)  dw\\
& \qquad + \sum_{|\beta| = \ell} \frac{1}{\beta!} \int (-w)^\beta D_v^k D^\beta_v\, g^\alpha \!\left(t, w, \frac{x}{t} - \xi \right)  dw - D^k_v F^{\alpha,\ell}_\infty \left(\frac{x}{t} - \xi \right)\\
& \qquad  + t^{-1}  \sum_{|\beta| = \ell + 1} \frac{1}{\beta!}\int (-w)^\beta D_v^k D^\beta_v \, g^\alpha \!\left(t, w, \frac{x - \theta^\alpha w}{t} - \xi \right)  dw  \biggr ] \ d\xi \biggr )
% & \lesssim \sum_{m=0}^{\ell-1} t^{\ell-m} \| \nabla_x \left( \Delta_x \right)^{-1} \nabla_x^k \rho_m(t) \|_\infty + t^{-1} \| \nabla_x \left( \Delta_x \right)^{-1} \nabla_x^k \rho_{\ell+1}(t) \|_\infty\\
% & \qquad  +  \frac{1}{4\pi} \int \frac{\xi_j}{|\xi|^3}  \sum_{\alpha = 1}^N q_\alpha \biggl ( \nabla_v^k F^{\alpha,\ell} \left(t, \frac{x}{t} - \xi \right)  - \nabla^k_vF^{\alpha,\ell}_\infty \left(\frac{x}{t} - \xi \right) \biggr ) \ d\xi \\
% & =: I + II + III.
\end{align*}
for some $\theta^\alpha \in [0,1]$. Next we estimate this expression. To estimate the first two terms on the right side, we use the definitions \eqref{eq:f-al} and \eqref{eq:rho-l} of $F^{\alpha,\ell}$ and $\rho_\ell$, respectively. To estimate the integral within the Taylor remainder term (the last on the right side) we use the uniform-in-time bound on the spatial support of $g^\alpha$ in \eqref{Sxg}. We thus obtain
\begin{align*}
	\left |t^{k+\ell+2} D^k_xE^j(t,x)  - D^k_v E^j_{\ell,\infty}\left(\frac{x}{t} \right) \right | & \lesssim \sum_{m=0}^{\ell-1} t^{\ell-m} \| \nabla_x \left( \Delta_x \right)^{-1} \nabla_x^k \rho_m(t) \|_\infty\\
	& \qquad  + \left |  \frac{1}{4\pi} \int \frac{\xi_j}{|\xi|^3}  \sum_{\alpha = 1}^N q_\alpha \biggl ( D_v^k F^{\alpha,\ell} \left(t, \frac{x}{t} - \xi \right)  - D^k_vF^{\alpha,\ell}_\infty \left(\frac{x}{t} - \xi \right) \biggr ) \ d\xi \right | \\
	& \qquad  + t^{-1} \max_{\alpha = 1, ..., N} \left\| \nabla_v^{k+\ell+1} g^\alpha(t) \right\|_\infty \\
	& =: I + II + III.
\end{align*}
To estimate $I$, we merely use Lemma \ref{LE} so that
$$I \lesssim \sum_{m=0}^{\ell-1} t^{\ell-m} \left\|\nabla_x^k \rho_m(t) \right\|_\infty.$$
We write the term $II$ as
$$II =  \left | \frac{1}{4\pi} \int \frac{\xi_j}{|\xi|^3} \mcA\left(t, \frac{x}{t} -\xi \right) d\xi \right |$$
where $\mcA(t,v)$ is defined by 
\begin{equation*}
\label{G1}
\mcA(t,v) =  \sum_{\alpha = 1}^N q_\alpha  \Big(D^k_v F^{\alpha,\ell}(t,v) - D^k_v F^{\alpha,\ell}_\infty(v) \Big).
\end{equation*}
Using Lemma \ref{LE} again, we find
$$II \lesssim \left \| \nabla_v (\Delta_v)^{-1} \mcA(t) \right \|_\infty  \lesssim \left \|  \mcA(t) \right \|_\infty \lesssim \max_{\alpha = 1, ..., N} \left\| \nabla^k_v F^{\alpha,\ell}(t)  - \nabla^k_v F^{\alpha,\ell}_\infty \right\|_\infty .$$
Lastly, the Taylor remainder term in $III$ satisfies
$$III \lesssim t^{-1} \mcG_v^{k+\ell+1}(t).$$
Finally, collecting these estimates and summing over all $k$th order derivatives yields the first result.
The proof of the second ($\ell =0$) result is identical with the exception that no Taylor expansion is required, and hence, the first term in the right side of the estimate does not appear.
\end{proof}

%\begin{lemma}
%\label{Drho}
%We have
%$$\sup_{x \in \bfR^3} \left | t \nabla_x\rho_0(t,x) - \nabla_v\rho_{0,\infty}\left(\frac{x}{t} \right) \right | \lesssim \max_{\alpha = 1, ..., N} \| \nabla_v F^{\alpha, 0}(t)  - \nabla_v F^{\alpha, 0}_\infty \|_\infty \lesssim t^{-1},$$
%and because $\rho_{0,\infty} \equiv 0$, this further implies
%$$\Vert \nabla_x\rho_0(t) \Vert_\infty \lesssim t^{-2}.$$
%\end{lemma}
%\begin{proof}
%Taking a derivative of \eqref{rho0} gives
%$$\nabla_x \rho_0(t,x)= t^{-1} \sum_{\alpha=1}^N q_\alpha  \int_{\mathbb{R}^3} \nabla_v g^\alpha \left(t,y,\frac{x}{t} \right)\,dy$$
%or, expressing this in terms of $F^{\alpha,0}(t,v)$,
%$$\nabla_x \rho_0(t,x)= t^{-1} \sum_{\alpha=1}^N q_\alpha \nabla_v F^{\alpha,0} \left(t, \frac{x}{t} \right).$$
%Hence, multiplying by $t$ and subtracting the proposed limit yields
%\begin{eqnarray*}
%\left | t \nabla_x \rho_0(t,x) - \nabla_v \rho_{0,\infty} \left( \frac{x}{t} \right) \right | 
%& = & \left | \sum_{\alpha=1}^N q_\alpha \left [ \nabla_v F^{\alpha,0} \left(t, \frac{x}{t} \right) -  \nabla_v F_\infty^{\alpha,0} \left(\frac{x}{t} \right)\right ] \right |\\
%& \lesssim & \max_{\alpha = 1, ..., N} \| \nabla_v F^{\alpha, 0}(t)  - \nabla_v F^{\alpha, 0}_\infty \|_\infty.
%\end{eqnarray*}
%\end{proof}

As an intermediate step to estimating field derivatives, we must also control derivatives of the approximations to the charge density.

\begin{lemma}
\label{Drhogeneral}
For any $k,\ell \in \N_0$, we have
$$\sup_{x \in \bfR^3} \left | t^k \nabla^k_x\rho_\ell(t,x) - \nabla^k_v\rho_{\ell,\infty}\left(\frac{x}{t} \right) \right | \lesssim \max_{\alpha = 1, ..., N} \| \nabla^k_v F^{\alpha, \ell}(t)  - \nabla^k_v F^{\alpha, \ell}_\infty \|_\infty$$
\end{lemma}
\begin{proof}
Taking any $k$th-order derivative of the representation for $\rho_\ell(t,x)$ in \eqref{eq:rho-l} gives
$$D^k_x \rho_\ell(t,x) =  t^{-k} \sum_{\alpha=1}^N q_\alpha D^k_v F^{\alpha, \ell} \left( t, \frac{x}{t} \right).$$
Hence, multiplying by $t^k$ and subtracting the proposed limit, given by the derivative of \eqref{eq:rho-l-inf}, yields
\begin{eqnarray*}
\left | t^k D^k_x \rho_\ell(t,x) - D^k_v \rho_{\ell,\infty} \left( \frac{x}{t} \right) \right | 
& = & \left | \sum_{\alpha=1}^N q_\alpha \left [ D^k_v F^{\ell,\alpha} \left(t, \frac{x}{t} \right) -  D^k_v F_\infty^{\ell,\alpha} \left(\frac{x}{t} \right)\right ] \right |\\
& \lesssim & \max_{\alpha = 1, ..., N} \| \nabla^k_v F^{\alpha, \ell}(t)  - \nabla^k_v F^{\alpha, \ell}_\infty \|_\infty,
\end{eqnarray*}
and upon summing over all such derivatives, this provides the stated result.
\end{proof}

The following lemma allows us to estimate higher-order field derivatives using fewer $v$-derivatives of $F^{\alpha,\ell}$ than within Lemma \ref{DEgeneral} with the addition of a $\ln(t)$ factor.

\begin{lemma}
\label{LDEgeneral}
For any $k,\ell \in \mathbb{N}$, we have
\begin{align*}
	\sup_{x \in \bfR^3} \left | t^{k+\ell+2} \nabla^k_xE(t,x) - \nabla^k_vE_{\ell,\infty}\left(\frac{x}{t} \right) \right | & \lesssim \sum_{m=0}^{\ell -1} t^{\ell-m} \|\nabla_x^k \rho_m(t) \|_\infty + \ln(t) \max_{\alpha = 1, ..., N}  \Vert \nabla_v^{k-1}F^{\alpha, \ell}(t) - \nabla_v^{k-1}F^{\alpha, \ell}_\infty\Vert_\infty \\
	& \qquad +  t^{-1} \left( \mcG_v^{k+\ell}(t) + \mcG_v^{k+\ell+1}(t) \right).
\end{align*}
Furthermore, in the case $\ell = 0$, we have
$$\sup_{x \in \bfR^3} \left | t^{k+2} \nabla^k_xE(t,x) - \nabla^k_vE_{0,\infty}\left(\frac{x}{t} \right) \right | \lesssim \ln(t) \biggl ( \max_{\alpha = 1, ..., N} \Vert \nabla_v^{k-1} F^{\alpha, 0}(t) -  \nabla_v^{k-1}F^{\alpha, 0}_\infty(t) \Vert_\infty +  t^{-1} \mcG_v^{k}(t) \biggr ).$$
for any $k \in \mathbb{N}$.
\end{lemma}
\begin{proof}
To prove the first result, we proceed as in the proof of Lemma \ref{DEgeneral}. Hence, we begin by taking any $k$th order derivative with respect to $x$ of the $j$th component of \eqref{E} and use the same Taylor series expansion of $g^\alpha$ so that
\begin{align*}
 t^{2+k+\ell}& D^k_xE^j(t,x) - D^k_v E^j_{\ell,\infty}\left(\frac{x}{t} \right) \\
& = \frac{1}{4\pi} \int \frac{\xi_j}{|\xi|^3}  \sum_{\alpha = 1}^N q_\alpha \left( t^\ell \int D^k_vg^\alpha \left(t, w, \frac{x-w}{t} - \xi \right)  dw - D^k_vF^{\alpha,\ell}_\infty \left(\frac{x}{t} - \xi \right) \right)\ d\xi\\
& = \frac{1}{4\pi} \int \frac{\xi_j}{|\xi|^3}  \sum_{\alpha = 1}^N q_\alpha \biggl [ \, \sum_{m=0}^{\ell-1} \sum_{|\beta| = m}  t^{\ell-m} \frac{1}{\beta!} \int (-w)^\beta D^k_vD^\beta_vg^\alpha \left(t, w, \frac{x}{t} - \xi \right)  dw\\
& \qquad \qquad \qquad \qquad + \sum_{|\beta| = \ell} \frac{1}{\beta!} \int (-w)^\beta D^k_v D^\beta_vg^\alpha \left(t, w, \frac{x}{t} - \xi \right)  dw - D^k_vF^{\alpha,\ell}_\infty \left(\frac{x}{t} - \xi \right)\\
& \qquad \qquad \qquad \qquad + t^{-1} \sum_{|\beta| = \ell+1} \frac{1}{\beta!} \int (-w)^\beta D^k_v D^\beta_v g^\alpha \left(t, w, \frac{x - \theta^\alpha w}{t}  - \xi \right)  dw  \biggr ] \ d\xi
\end{align*}
for some $\theta^\alpha \in [0,1]$. Using the bounded spatial support of $g^\alpha$ established in \eqref{Sxg} to estimate the integral within the Taylor remainder term, this becomes
\begin{align*}
	\left |t^{2+k+\ell} D^k_xE^j(t,x)  - D^k_v E^j_{\ell,\infty}\left(\frac{x}{t} \right) \right | & \lesssim \sum_{m=0}^{\ell-1} t^{\ell-m} \| \nabla_x \left( \Delta_x \right)^{-1} \nabla_x^k \rho_m(t) \|_\infty\\
	& \qquad  + \left |  \frac{1}{4\pi} \int \frac{\xi_j}{|\xi|^3}  \sum_{\alpha = 1}^N q_\alpha \biggl ( D_v^k F^{\alpha,\ell} \left(t, \frac{x}{t} - \xi \right)  - D^k_vF^{\alpha,\ell}_\infty \left(\frac{x}{t} - \xi \right) \biggr ) \ d\xi \right | \\
	& \qquad  + t^{-1} \max_{\alpha = 1, ..., N} \| \nabla_v^{k+\ell+1} g^\alpha(t) \|_\infty \\
	& =: I + II + III.
\end{align*}

To estimate $I$, we merely use Lemma \ref{LE} so that
$$I \lesssim \sum_{m=0}^{\ell-1} t^{\ell-m} \|\nabla_x^k \rho_m(t) \|_\infty.$$
The Taylor remainder term in $III$ merely satisfies 
$$III \lesssim  t^{-1} \mcG_v^{k+\ell+1}(t).$$
Lastly, the limiting term within $II$ is transformed into
$$II =  \left |\frac{1}{4\pi} \int \frac{\xi_j}{|\xi|^3} \partial_{v_i}\mcA\left(t, \frac{x}{t} -\xi \right) \ d\xi \right |$$
for some $i = 1, 2, 3$, where $\mcA(t,v)$ is defined by 
\begin{equation*}
\label{G1}
\mcA(t,v) =  \sum_{\alpha = 1}^N q_\alpha \left( D^{k-1}_v F^{\alpha,\ell}(t,v) - D^{k-1}_vF^{\alpha,\ell}_\infty(v) \right)
\end{equation*}
and the $k$th order derivative is decomposed via $D_v^k = \partial_{v_i} D_v^{k-1}$.
From this definition we first note
\begin{equation}
\label{DA1}
\| \partial_{v_i}\mcA(t) \|_\infty \lesssim \mcG_v^{k+\ell}(t)
\end{equation}
and, because each $F^{\alpha, \ell}_\infty(v)$ is compactly supported and each $g^\alpha(t,x,v)$ has bounded support both in $x$, due to \eqref{Sxg}, and in the phase space $(x,v)$ from Lemma \ref{Sg}, we find
$$\Vert \mcA(t) \Vert_1 + \Vert \mcA(t) \Vert_\infty \lesssim \max_{\alpha = 1, ..., N} \Vert \nabla^{k-1}_vF^{\alpha,\ell}(t) - \nabla^{k-1}_vF^{\alpha,\ell}_\infty \Vert_\infty.$$
Then, decomposing the difference of derivatives and using the identity
$$\partial_{v_i}\mcA \left(t,  \frac{x}{t} - \xi \right) = -  \partial_{\xi_i}\left [ \mcA \left(t,  \frac{x}{t} - \xi \right) \right ]$$ 
with an integration by parts away from the singularity, letting $0<d<R<\infty$ to be chosen later, we find
\begin{eqnarray*}
II & \lesssim & \left | \int_{|\xi| < d} \frac{\xi_j}{|\xi|^3} \partial_{v_i} \mcA\left(t,\frac{x}{t}-\xi \right) \ d\xi \right | +  \left | \int_{|\xi| = d} \frac{\xi_i\xi_j}{|\xi|^4} \mcA \left(t,\frac{x}{t}-\xi \right)  \ dS_\xi \right |\\
& &  + \left | \int_{d < |\xi| < R}  \partial_{\xi_i}  \left( \frac{\xi_j}{|\xi|^3} \right) \mcA\left(t,\frac{x}{t}-\xi \right) \ d\xi \right | + \left | \int_{|\xi| > R} \partial_{\xi_i}  \left( \frac{\xi_j}{|\xi|^3} \right) \mcA\left(t,\frac{x}{t}-\xi \right)  \ d\xi \right |\\
& =: &  II_A + II_B + II_C + II_D.
\end{eqnarray*}
The estimate of $II_A$ merely uses \eqref{DA1} so that
$$II_A \lesssim d \mcG_v^{k+\ell}(t).$$
We use the previous $L^\infty$ bound to estimate $II_B$, which yields
$$II_B \lesssim \| \mcA(t) \|_\infty \left( \int_{|\xi| = d} |\xi|^{-2} dS_\xi \right) \lesssim  \| \mcA(t) \|_\infty \lesssim  \max_{\alpha = 1, ..., N} \Vert \nabla^{k-1}_vF^{\alpha,\ell}(t) - \nabla^{k-1}_vF^{\alpha,\ell}_\infty \Vert_\infty.$$
Similarly, this provides an estimate of $II_C$ as
\begin{eqnarray*}
II_C & \lesssim & \int_{d < |\xi| < R} |\xi|^{-3}  \left |\mcA \left (t,\frac{x}{t}-\xi \right) \right | \ d\xi\\
& \lesssim & \| \mcA(t) \|_\infty \ln \left(\frac{R}{d}\right) \lesssim \ln \left(\frac{R}{d}\right) \max_{\alpha = 1, ..., N} \Vert \nabla^{k-1}_vF^{\alpha,\ell}(t) - \nabla^{k-1}_vF^{\alpha,\ell}_\infty \Vert_\infty.
\end{eqnarray*}
Finally, we use the $L^1$ bound to estimate $II_D$ so that
$$II_D \lesssim \int_{|\xi| > R} |\xi|^{-3}  \left |\mcA \left (t,\frac{x}{t}-\xi \right) \right | \ d\xi \lesssim R^{-3} \| \mcA(t) \|_1 \lesssim R^{-3} \max_{\alpha = 1, ..., N} \Vert \nabla^{k-1}_vF^{\alpha,\ell}(t) - \nabla^{k-1}_vF^{\alpha,\ell}_\infty \Vert_\infty.$$
Collecting these estimates gives
$$II \lesssim d \mcG_v^{k+\ell}(t) + \left( 1 + \ln \left(\frac{R}{d} \right) + R^{-3} \right) \max_{\alpha = 1, ..., N} \Vert \nabla^{k-1}_vF^{\alpha,\ell}(t) - \nabla^{k-1}_vF^{\alpha,\ell}_\infty \Vert_\infty.$$
Choosing $d = t^{-1}$ with $R^{-3} = \ln(t)$ implies $\ln \left(\frac{R}{d} \right) \lesssim \ln(t)$ and, upon including the estimates for $I$ and $III$, yields the first stated result.

The second result is nearly identical with the exception that no Taylor expansion is required.
Thus, taking any $k$th order derivative with respect to $x$ of the $j$th component of \eqref{E} as above, we find
\begin{align*}
t^{k+2} D^k_xE^j(t,x)  - D^k_v E^j_{0,\infty}\left(\frac{x}{t} \right)
&= \frac{1}{4\pi} \int \frac{\xi_j}{|\xi|^3}  \sum_{\alpha = 1}^N q_\alpha \left(\int D^k_vg^\alpha \left(t, w, \frac{x-w}{t} - \xi \right)  dw - D^k_vF^{\alpha,0}_\infty \left(\frac{x}{t} - \xi \right) \right) d\xi\\
& = \frac{1}{4\pi} \int \frac{\xi_j}{|\xi|^3}  \sum_{\alpha = 1}^N q_\alpha \int D^k_v \left [ g^\alpha \left(t, w, \frac{x-w}{t} - \xi \right) - g^\alpha \left(t, w, \frac{x}{t} - \xi \right)\right ]  dw \, d\xi\\
& \qquad  \qquad + \frac{1}{4\pi} \int \frac{\xi_j}{|\xi|^3}  \sum_{\alpha = 1}^N q_\alpha \left [D^k_v F^{\alpha,0} \left(t, \frac{x}{t} - \xi \right) - D^k_vF^{\alpha,0}_\infty \left(\frac{x}{t} - \xi \right) \right ]\ d\xi\\
& = \frac{1}{4\pi} \int \frac{\xi_j}{|\xi|^3} \left [ \partial_{v_i} \mcA_1\left(t,\frac{x}{t}-\xi \right) + \partial_{v_i}\mcA_2 \left(t,\frac{x}{t}-\xi \right) \right ] \ d\xi
\end{align*}
for some $i = 1, 2, 3$ where $\mcA_1(t,v)$ and $\mcA_2(t,v)$ are defined by 
\begin{equation*}
\label{G1_1}
\mcA_1(t,v) =  \sum_{\alpha = 1}^N q_\alpha \int D_v^{k-1}\left [g^\alpha \left(t, w, v  - \frac{w}{t} \right) - g^\alpha(t,w, v) \right ] dw,
\end{equation*}
and
\begin{equation*}
\label{G1_2}
\mcA_2(t,v) = \sum_{\alpha = 1}^N q_\alpha D^{k-1}_{v}\left [F^{\alpha,0}(t,v) - F^{\alpha,0}_\infty(v) \right ],
\end{equation*}
respectively. 
As before, the $k$th order derivative satisfies $D_v^k = \partial_{v_i} D_v^{k-1}$.
From these definitions we note that
$$\| \partial_{v_i} \mcA_1(t) \|_\infty + \Vert \partial_{v_i}\mcA_2(t) \|_\infty \lesssim \mcG_v^k(t).$$
To control the $\mcA_1$ term, we use the uniform-in-time bound on the spatial support of $g^\alpha$ from \eqref{Sxg} and velocity derivative estimates, which yields
\begin{align*}
\| \mcA_1(t) \|_\infty
& = \sup_{v \in \bfR^3} \left | \sum_{\alpha = 1}^N q_\alpha \int D^{k-1}_v \left [g^\alpha \left(t, w, v  - \frac{w}{t} \right) -  g^\alpha(t,w, v) \right ] \ dw \right |\\
& \lesssim  \sup_{v \in \bfR^3} \sum_{\alpha = 1}^N \int \left |\int_0^1 \frac{d}{d\theta} \left(D^{k-1}_v g^\alpha \left(t, w, v - \theta \frac{w}{t} \right)  \right)  d\theta \right |  dw\\
& \lesssim  t^{-1} \sup_{v \in \bfR^3} \sum_{\alpha = 1}^N \int_0^1 \int |w| \left |\nabla_v D^{k-1}_v g^\alpha \left(t, w, v - \theta \frac{w}{t} \right) \right |  dw d\theta\\
& \lesssim t^{-1} \mcG_v^k(t).
\end{align*}
Because of the bounded phase space support of $g^\alpha(t,x,v)$ from Lemma \ref{Sg}, we further find
$$\| \mcA_1(t) \|_1  \lesssim t^{-1} \mcG_v^k(t).$$
Additionally, because each $g^\alpha(t,x,v)$ has bounded support both in $x$ and in the phase space $(x,v)$ and $F^{\alpha, 0}_\infty$ has compact support, we find
$$\Vert \mcA_2(t) \Vert_1 + \Vert \mcA_2(t) \Vert_\infty \lesssim \max_{\alpha = 1, ..., N} \Vert \nabla^{k-1}_vF^{\alpha,0}(t) - \nabla^{k-1}_vF^{\alpha,0}_\infty \Vert_\infty.$$
Replacing $\mcA(t,v)$ with $\mcA_1(t,v) + \mcA_2(t,v)$ within the decomposition of $II$ above and using these estimates then gives the stated result.
\end{proof}

\subsection{Improved Bounds on the Particle Density}

The final lemma provides control of the charge density in terms of previous approximations and derivatives of $g^\alpha$.

\begin{lemma}
\label{LDensitygeneral}
For any $\ell \in \mathbb{N}$, we have
$$\sup_{x \in \bfR^3} \left | t^{\ell+3} \rho(t,x) - \rho_{\ell,\infty} \left(\frac{x}{t} \right) \right | \lesssim \sum_{j=0}^{\ell-1}  t^{\ell - j}\left \|  \rho_j(t) \right \|_\infty + \max_{\alpha = 1,...,N} \left \| F^{\alpha,\ell}(t) - F^{\alpha,\ell}_\infty \right \|_\infty  + t^{-1} \mcG_v^{\ell +1}(t) .$$
\end{lemma}
\begin{proof}
Beginning with \eqref{rho}, which yields 
$$\rho(t,x) = \frac{1}{t^3}  \sum_{\alpha = 1}^N q_\alpha \int g^\alpha \left(t, y, \frac{x-y}{t} \right) \ dy,$$
we merely use the Taylor expansion \eqref{g:taylor} and the definitions \eqref{eq:f-al} and \eqref{eq:rho-l} to write
$$t^3 \rho(t,x) = \sum_{j=0}^{\ell-1} t^{-j}\rho_j(t,x) + t^{-\ell} \rho_\ell(t,x) + t^{-\ell-1} \mathcal{R}(t)$$
where $\left|\mathcal{R}(t)\right|\lesssim \mcG_v^{\ell+1}(t)$.
By the definitions in \eqref{eq:rho-l} and \eqref{eq:rho-l-inf}, the $\ell$th order approximation of $\rho(t,x)$ satisfies
$$\left | \rho_\ell(t,x) - \rho_{\ell,\infty} \left( \frac{x}{t} \right) \right | \lesssim \max_{\alpha = 1,...,N} \left | F^{\alpha, \ell} \left( t, \frac{x}{t} \right) - F^{\alpha, \ell}_\infty \left(\frac{x}{t} \right) \right |.$$
Thus, expanding $\rho(t,x)$ as above, we find
\begin{align*}
\left | t^{\ell+3} \rho(t,x) - \rho_{\ell,\infty} \left(\frac{x}{t} \right) \right | &\leq  \left | \sum_{j=0}^{\ell -1} t^{\ell - j} \rho_j(t,x) +  \rho_{\ell}(t,x) - \rho_{\ell,\infty} \left( \frac{x}{t} \right) \right | + t^{-1} \mcG_v^{\ell+1}(t) \\ 
&\lesssim  \sum_{j=0}^{\ell-1} t^{\ell - j}\left \|  \rho_j(t) \right \|_\infty + \left | \rho_\ell(t,x) - \rho_{\ell,\infty} \left( \frac{x}{t} \right) \right | + t^{-1} \mcG_v^{\ell+1}(t)\\
&\lesssim  \sum_{j=0}^{\ell-1} t^{\ell - j}\left \|  \rho_j(t) \right \|_\infty + \max_{\alpha = 1,...,N} \left | F^{\alpha, \ell} \left( t, \frac{x}{t} \right) - F^{\alpha, \ell}_\infty \left(\frac{x}{t} \right) \right | + t^{-1} \mcG_v^{\ell+1}(t),
\end{align*}
which completes the proof.
\end{proof}

\section{Proofs of Theorems}
\label{sec:proofs}

First, we establish Theorem \ref{T2}, which will be used later to prove Theorem \ref{T1}.

\begin{proof}[Proof of Theorem \ref{T2}]
We will prove the result by induction. 
In particular,  we consider the statement of the proof to be of the form - for every $n \in \N_0$, $P_n \Rightarrow Q_n$ where
$P_n$ represents the statement
$$\forall \ell=0,\dots,n,\qquad\rho_{\ell,\infty}\equiv 0,$$
while $Q_n$ represents the collection of statements
$$ \left \{ \,
\begin{aligned}
\sup_{x \in \mathbb{R}^3} \left | t^{n+4}\rho(t) - \rho_{n+1, \infty} \left( \frac{x}{t} \right) \right | & \lesssim t^{-1},\\
\sup_{x \in \mathbb{R}^3} \left |t^{n+3} E(t) - E_{n+1,\infty} \left( \frac{x}{t} \right) \right | & \lesssim t^{-1},\\
\sup_{x \in \mathbb{R}^3} \left |t^{n+3}\nabla^k_xE(t) - \nabla_v^k E_{n+1-k,\infty} \left( \frac{x}{t} \right) \right | & \lesssim t^{-1}, \qquad \forall k =1, ..., n+1 \\
\sup_{x \in \mathbb{R}^3} \left |t^{n+4}\nabla^{n+2}_xE(t) - \nabla_v^{n+2}E_{0,\infty} \left( \frac{x}{t} \right) \right | & \lesssim t^{-1}\ln(t),\\
 \mcG^k_{x,v}(t) + \mcG^k_{v}(t) & \lesssim 1, \qquad \forall k = 1, ..., n+2,
\end{aligned}
\right. $$
%where $P_n$ and $Q_n$ are as follows:
%\begin{center}
%	\begin{tabular}{ |l|l|} 
%	 \hline
%	\multicolumn{1}{|c|}{$P_n$} &	 \multicolumn{1}{|c|}{$Q_n$}  \\ 
%	 &\\
%	 $\exists p \in \left(\frac{5}{3}, 2 \right]$ s.t. $\|E(t) \|_\infty \lesssim t^{-p}$ & $\| t^{n+4}\rho(t) - \rho_{n+1, \infty} \|_\infty \lesssim t^{-1}$\quad \text{and}  \quad$\| t^{n+3} E(t) - E_{n+1,\infty} \|_\infty \lesssim t^{-1}$ \\
%$\mathcal{M}=0$
%&$\forall k \in \{1, ..., n+1\}$: \quad$\| t^{n+3}D^k_xE(t) - D_v^k E_{n+1-k,\infty} \|_\infty \lesssim t^{-1} ,$\\ 
%& \qquad\qquad\qquad\qquad\qquad$\| t^{n+4}D^{n+2}_xE(t) - D_v^{n+2}E_{0,\infty} \|_\infty \lesssim t^{-1}\ln(t) ,$\\ 
%	$\forall k \in \{0,1,\dots,n\}$:\qquad $\rho_{k,\infty}\equiv0$& $\forall k \in \{1, ..., n+2\}$: \quad$\mcG^k_{x,v}(t) + \mcG^k_{v}(t) \lesssim 1$ \\
%	 \hline
%	\end{tabular}
%	\end{center}
where we remind the reader that $\mcG^k_{v}$ and $\mcG^k_{x,v}$ are defined in \eqref{eq:G^k_v} and \eqref{eq:G^k_x,v}, respectively.
Then, the proof requires that we show
\begin{enumerate}
\item Base case: $P_0 \Rightarrow Q_0$,
\item Inductive step: assume $P_{n-1} \Rightarrow Q_{n-1}$ and show $P_n \Rightarrow Q_n$.
\end{enumerate}

\subsection{Base case}
To prove the base case, we assume $\rho_{0,\infty} \equiv 0$, which further implies $E_{0,\infty} \equiv 0$. 
Due to Theorem \ref{oldT2} this assumption provides the estimates
\begin{equation}
\label{BC}
\left \{ \,
\begin{aligned}
\| F^{\alpha,0}(t) - F^{\alpha,0}_\infty \|_\infty & \lesssim t^{-2}, \qquad \alpha = 1, ..., N\\
\| \rho(t) \|_\infty & \lesssim t^{-4},\\
\| E(t) \|_\infty & \lesssim t^{-3}, \\ 
\| \nabla_xE(t) \|_\infty & \lesssim t^{-4} \ln(t),\\
\mcG_v^1(t) +  \mcG_{x,v}^1(t) & \lesssim 1.
\end{aligned}
\right .
\end{equation}

With this improved decay (and applying Lemma \ref{LDE} with $k=2$), we construct initial estimates on 2nd-order derivative terms with a lemma whose proof is postposed until the following section.
\begin{lemma}
\label{Dng0}
Assume  that $\rho_{0,\infty}\equiv 0$, which implies \eqref{BC}.
Then, we have
%\begin{equation}
%\label{DngDnE}
%\mcG_v^{n+2}(t) \lesssim 1 + \int_1^t \biggl[ s^{n+3} \|D_x^{n+2}E(s)\|_{\infty} + s \Vert D_xE(s) \Vert_\infty \left( s  \mcG_{x,v}^{n+2}(s)  + \mcG_v^{n+2}(s) \right)\biggr ] \ ds,
%\end{equation}
\begin{equation}
\label{DngDnE0}
\mcG_v^{2}(t) \lesssim 1 + \int_1^t s^{3} \|\nabla_x^{2}E(s)\|_{\infty} \ ds,
\end{equation}
and, as a preliminary estimate,
$$\Vert \nabla_x^{2} E(t) \Vert_\infty \lesssim t^{-4} \ln(t), \qquad \mcG_{x,v}^{2}(t) \lesssim 1, \qquad \mathrm{and} \qquad  \mcG_v^{2}(t) \lesssim \ln^2(t).$$
\end{lemma}
\begin{proof}
The proof is contained in Section \ref{sec:bounds-derivs-g} below.
\end{proof}

%From Lemma \ref{LDE} with $k=2$ we know that
%$$\| \nabla_x^2E(t) \|_\infty \lesssim t^{-4} \left [ 1 + \ln^* \left( \mcG_v^2(t) \right) \right ].$$
%Combining this estimate with the results of Lemma \ref{Dng0}, we find
%$$
%\begin{gathered}
%\| \nabla_x^2E(t) \|_\infty \lesssim t^{-4} \ln(t),\\
%\mcG_{x,v}^2(t) \lesssim 1, \\
% \mcG_v^2(t) \lesssim \ln^2(t).
%\end{gathered}
%$$

With these preliminary estimates in hand, we now establish the convergence of $\nabla_vF^{\alpha,0}(t,v)$.
First, using Lemma \ref{dvkrho} with $k=1$ and \eqref{BC} we find
$$\| \nabla_x \rho(s)\|_\infty \lesssim t^{-4}\mcG_v^1(t) \lesssim t^{-4}.$$
Then, invoking Lemma \ref{dvkFalpha0} with $k=1$ and using this together with  \eqref{BC} and the preliminary estimate of $\mcG_v^2(t)$ in Lemma \ref{Dng0} we have
\begin{align*}
\infnorm{\nabla_v F^{\alpha,0}(t) - \nabla_v F^{\alpha,0}_\infty} 
& \lesssim \int_t^\infty \Big [s^2 \| \nabla_x \rho(s)\|_\infty + s \mcG_v^1(s)\big(\infnorm{\rho(s)}+\infnorm{\nabla_x E(s)}\big) + \infnorm{E(s)} \mcG_v^2(s) \Big ] ds\\
& \lesssim \int_t^\infty s^{-2} \ ds  \lesssim t^{-1}.
\end{align*}
Hence, the function $F^{\alpha,0}(t,v)$
satisfies
\begin{equation}
\label{dF0}
\| \nabla_v F^{\alpha,0}(t) - \nabla_vF^{\alpha,0}_\infty \|_\infty \lesssim t^{-1}
\end{equation}
for any $\alpha = 1, ..., N$,
which provides the existence and regularity of the limiting function $\nabla_v F^{\alpha,0}_\infty(v)$, as mentioned within Remark \ref{Finf}.

%Therefore, using Lemma \ref{Drhogeneral} with $k=1$ and $\ell = 0$, we have
%$$ \sup_{x \in \bfR^3} \left | t \nabla_x \rho_0(t,x)  - \nabla_v \rho_{0,\infty}\left( \frac{x}{t} \right) \right |  \lesssim t^{-1}.$$ 
%As $\rho_{0,\infty} \equiv 0$, this yields
%$$ \|\nabla_x \rho_0(t) \|_\infty  \lesssim t^{-2}.$$ 

Next, we use this to improve the preliminary estimate of second-order field derivatives with Lemma \ref{LDEgeneral}. In particular, we note that this result provides an estimate in terms of $\nabla_v F$, while use of Lemma \ref{DEgeneral} would require us to first estimate
$\nabla_v^2 F$.
With the convergence result for $\nabla_v F^{\alpha,0}(t,v)$ in place, we use Lemma \ref{LDEgeneral} with $k=2$ and $\ell = 0$, the preliminary estimate of $\mcG_v^2(t)$ in Lemma \ref{Dng0}, and \eqref{dF0} to conclude 
$$\sup_{x \in \mathbb{R}^3} \left | t^4 \nabla_x^2E(t,x) - \nabla_v^2 E_{0,\infty} \left( \frac{x}{t} \right) \right | \lesssim \ln(t) \biggl ( \Vert \nabla_v F^{\alpha, 0}(t) -  \nabla_v F^{\alpha, 0}_\infty(t) \Vert_\infty +  t^{-1} \mcG_v^2(t) \biggr ) \lesssim t^{-1}\ln^3(t).$$
Because the limiting function vanishes, i.e. $E_{0,\infty} \equiv 0$, this implies a refined decay rate of second-order field derivatives, namely
$$\| \nabla_x^2E(t) \|_\infty \lesssim t^{-5}\ln^3(t).$$
This faster decay then implies the uniform bound $\mcG_v^2(t) \lesssim 1$ 
by revisiting \eqref{DngDnE0},
so that
\begin{equation}
\label{Gv2}
\mcG_{x,v}^2(t) + \mcG_v^2(t) \lesssim 1.
\end{equation}
This further yields
\begin{equation}
\label{Dx2E}
\sup_{x \in \mathbb{R}^3} \left | t^4 \nabla_x^2E(t,x) - \nabla_v^2 E_{0,\infty} \left( \frac{x}{t} \right) \right | \lesssim t^{-1}\ln(t)
\end{equation}
by using \eqref{Gv2} within the above estimate.
As before, because the limiting function $E_{0,\infty}$ vanishes, the second-order derivatives of the electric field then satisfy
$$\| \nabla_x^2E(t) \|_\infty \lesssim t^{-5}\ln(t).$$

Next, we improve the decay estimate on first-order field derivatives. In this direction, invoking Lemma \ref{DEgeneral} with $k=1$ and $\ell = 0$ and using \eqref{dF0} and \eqref{Gv2} implies
\begin{equation}
\label{Dx1E}
\sup_{x \in \bfR^3} \left | t^3 \nabla_xE(t,x) - \nabla_vE_{0,\infty}\left(\frac{x}{t} \right) \right | \lesssim \Vert \nabla_v F^{\alpha, 0}(t) -  \nabla_v F^{\alpha, 0}_\infty(t) \Vert_\infty +  t^{-1} \mcG_v^2(t)  \lesssim t^{-1},
\end{equation}
which, due to $E_{0,\infty} \equiv 0$, gives
$$ \Vert \nabla_xE(t) \Vert_\infty \lesssim t^{-4}.$$ 

This improved decay now allows us to sharply estimate the convergence rate of $F^{\alpha,1}(t,v)$. 
Hence, we invoke Lemma \ref{dvkFell} with $k=0$ and $\ell = 1$, the established decay rates of the field and field derivatives, and \eqref{Gv2} so that the function
$$F^{\alpha,1}(t,v) = \int (-x) \cdot \nabla_v g^\alpha(t,x,v) \ dx$$
satisfies
\begin{align*}
\| F^{\alpha,1}(t) - F^{\alpha,1}_\infty \|_\infty & \lesssim \int_t^\infty  \left [s \Vert \nabla_x E(s) \Vert_\infty \biggl ( s\mcG_{x,v}^{1}(s) + \mcG_v^{1}(s) \biggr) + \Vert E(s) \Vert_\infty \biggl ( s\mcG_{x,v}^{2}(s) + \mcG_v^{2}(s) \biggr) \right ] ds\\
& \lesssim  \int_t^\infty  s^{-2} \ ds  \lesssim t^{-1}
\end{align*}
for any $\alpha = 1, ..., N$,
which establishes the existence of the limiting function $F^{\alpha,1}_\infty(v)$, as mentioned within Remark \ref{Finf}.
%where for all $v \in \bfR^3$
%$$F^{\alpha,1}_\infty(v) =  \lim_{t\to\infty} F^{\alpha,1}(t,v) = \int  (-x)\cdot \nabla_v f^\alpha_\infty(x,v) \ dx.$$
Using this quantity, we define the next order limiting charge density and field via
$$\rho_{1,\infty}(v) = \sum_{\alpha = 1}^N q_\alpha F_\infty^{1,\alpha}(v)
\qquad \mathrm{and} \qquad
E_{1,\infty}(v) =\nabla_v (\Delta_v)^{-1} \rho_{1,\infty}(v).$$

Prior to using these limits to obtain the next order of convergence in the charge density and electric field, 
we invoke Lemma \ref{Drhogeneral} with $k=0$ and $\ell = 0$ and \eqref{BC} to find
$$\left|{\rho_0(t,x)-\rho_{0,\infty}\rndP{\frac{x}{t}}}\right|\lesssim \max_{\alpha=1,\hdots,N}\infnorm{F^{\alpha,0}(t)-F_\infty^{\alpha,0}} \lesssim t^{-2}$$
so that, due to the vanishing of the limit $\rho_{0,\infty}$, we have
$$\Vert \rho_0(t) \Vert_\infty \lesssim t^{-2}.$$
%
%Additionally, we use Lemma \ref{LE} with this estimate so that 
%$$ \Vert E_0(t)\Vert_\infty \lesssim \Vert \nabla_x{\rndP{\Delta_x}}^{-1}\rho_0(t) \Vert_\infty \lesssim \infnorm{\rho_0(t)} \lesssim t^{-2}.$$ 
Finally, using the proposed limits for $\rho_1(t,x)$ and $E_1(t,x)$, the estimates on $\Vert \rho_0(t) \Vert_\infty$ and $\| F^{\alpha,1}(t) - F^{\alpha,1}_\infty \|_\infty$, and \eqref{Gv2} within Lemma \ref{DEgeneral} with $k =0$ and $\ell = 1$ and Lemma \ref{LDensitygeneral} with $\ell = 1$ gives
\begin{equation}
\label{rho1}
\sup_{x \in \mathbb{R}^3} \left | t^4\rho(t,x) - \rho_{1, \infty} \left( \frac{x}{t} \right) \right | \lesssim t \Vert \rho_0(t) \Vert_\infty +  \max_{\alpha = 1,...,N} \left \| F^{\alpha,1}(t) - F^{\alpha,1}_\infty \right \|_\infty  + t^{-1} \mcG_v^{2}(t) \lesssim t^{-1}
\end{equation}
and
\begin{equation}
\label{E1}
\sup_{x \in \mathbb{R}^3} \left | t^3 E(t,x) - E_{1,\infty}\left( \frac{x}{t} \right) \right | \lesssim  t \Vert \rho_0(t) \Vert_\infty + \max_{\alpha = 1,...,N} \left \| F^{\alpha,1}(t) - F^{\alpha,1}_\infty \right \|_\infty  + t^{-1} \mcG_v^{2}(t) \lesssim t^{-1}.
\end{equation}

Assembling the estimates \labelcref{Gv2,Dx2E,Dx1E,rho1,E1} 
% \eqref{Gv2}, \eqref{Dx2E}, \eqref{Dx1E}, \eqref{rho1}, and \eqref{E1}
then implies $Q_0$ and completes the base case.
%$$\begin{gathered}
%\| t^4\rho(t) - \rho_{1, \infty} \|_\infty \lesssim t^{-1},\\
%\| t^3 E(t) - E_{1,\infty} \|_\infty \lesssim t^{-1},\\
%\| t^3\nabla_xE(t) - \nabla_v E_{0,\infty} \|_\infty \lesssim t^{-1},\\ 
%\| t^4D^2_xE(t) - D_v^2E_{0,\infty} \|_\infty \lissome t^{-1}\ln(t)\\
%\mcG^1_{x,v}(t) + \mcG^1_{v}(t) \lesssim 1\\
%\mcG^2_{x,v}(t) + \mcG^2_{v}(t) \lesssim 1,
%\end{gathered}$$
%which have all been established.

%%%%%%%%%%%%%%%%%%%%%%%%%%%%%%%%%%

\subsection{Inductive step}
To address the inductive step we fix $n \geq 1$ and assume $P_{n-1} \Rightarrow Q_{n-1}$ and $P_n$, so that we must then establish $Q_n$.
Note that, as $P_n$ implies $P_{n-1}$, we can immediately deduce $Q_{n-1}$, as well.
Thus, by the induction hypothesis
$$
\begin{gathered}
\sup_{x \in \mathbb{R}^3} \left | t^{n+3}\rho(t) - \rho_{n, \infty} \left( \frac{x}{t} \right) \right | \lesssim t^{-1}, \\
\sup_{x \in \mathbb{R}^3} \left |  t^{n+2} E(t) - E_{n,\infty} \left( \frac{x}{t} \right) \right | \lesssim t^{-1},
\end{gathered}
$$
the lower order field estimates
$$\sup_{x \in \mathbb{R}^3} \left |  t^{n+2}\nabla^k_xE(t) - \nabla_v^k E_{n-k,\infty} \left( \frac{x}{t} \right) \right | \lesssim t^{-1}, \qquad \forall k \in \{1, ..., n\},$$
and
the higher order field estimate
$$\sup_{x \in \mathbb{R}^3} \left |  t^{n+3}\nabla^{n+1}_xE(t) - \nabla_v^{n+1} E_{0,\infty} \left( \frac{x}{t} \right) \right | \lesssim t^{-1}\ln(t)$$
all hold.
Because $\rho_{n,\infty} \equiv 0$ and $E_{k,\infty} \equiv 0$ for all $k = 0, ..., n$, the above quantities enjoy faster rates of decay, and thus we have the following set of estimates stemming directly from the induction hypothesis
\begin{equation}
	\label{IH}
	\left \{ \,
	\begin{aligned}
		\| \rho(t) \|_\infty & \lesssim t^{-n-4}, & \\
		\| E(t)  \|_\infty & \lesssim t^{-n-3}, & \\
		\| \nabla^k_xE(t) \|_\infty & \lesssim t^{-n-3} \qquad & \forall k \in \{1, ..., n\}, & \\
		\| \nabla^{n+1}_xE(t) \|_\infty & \lesssim t^{-n-4}\ln(t),\\
		\mcG^k_{x,v}(t) + \mcG^k_{v}(t) & \lesssim 1 \qquad & \forall k \in \{1, ..., n\},\\
		\mcG^{n+1}_{x,v}(t) + \mcG^{n+1}_{v}(t) & \lesssim 1. & 
	\end{aligned}
	\right .
\end{equation}
This represents the starting point for the next iteration and is analogous to the decay rates \eqref{BC} inherited in the base case.
%Many of these terms possess the expected upper bound behavior, but we must show that they converge to limits $\rho_{n+1,\infty}$ and $E_{n+1,\infty}$ with the rate listed on the right side.
For clarity, we  separate the  inductive step into several smaller steps.\\

\emph{\underline{Step 1:} Preliminary estimates of highest order derivatives}

We begin by stating the following lemma, which generalizes Lemma \ref{Dng0} for the induction step.

\begin{lemma}
\label{Dng}
Let $n \in \N$ be given and assume that $\rho_{k,\infty} \equiv 0$ for all $k = 0, ..., n$, which implies \eqref{IH}.
Then, we have
\begin{equation}
\label{DngDnE}
\mcG_v^{n+2}(t) \lesssim 1 + \int_1^t s^{n+3} \left\|\nabla_x^{n+2}E(s)\right\|_{\infty} \ ds,
\end{equation}
and, as a preliminary estimate,
$$\left\Vert \nabla_x^{n+2} E(t) \right\Vert_\infty \lesssim t^{-4-n} \ln(t), \qquad \mcG_{x,v}^{n+2}(t) \lesssim 1, \qquad \mathrm{and} \qquad  \mcG_v^{n+2}(t) \lesssim \ln^2(t).$$
\end{lemma}
\begin{proof}
The proof is contained in Section \ref{sec:bounds-derivs-g} below.
\end{proof} 
Though the proof of this lemma is postponed until the final section, we must note that Lemma \ref{Dng} relies upon Lemma \ref{LDE} in the same way that Lemma \ref{Dng0} invoked this result, namely to obtain an estimate of the highest order field derivatives in terms of the highest order derivatives of the translated particle distribution.

%Due to the bound on $\mcG^{n+1}_{v}(t)$ stemming from the induction hypothesis, we first invoke Lemma \ref{LDE} with $k=n+2$, which yields
%$$\| \nabla_x^{n+2}E(t) \|_\infty \lesssim t^{-4-n} \left [ 1 + \ln^* \left(\mcG_v^{n+2}(t) \right) \right ], $$
%and combining this estimate with the result of Lemma \ref{Dng}, we find
%$$
%\begin{gathered}
%\| \nabla_x^{n+2}E(t) \|_\infty \lesssim t^{-4-n} \ln(t),\\
%\mcG_{x,v}^{n+2}(t) \lesssim 1, \\
% \mcG_v^{n+2}(t) \lesssim \ln^2(t).
%\end{gathered}
%$$

\vspace{0.1in}

\emph{\underline{Step 2:} Derivative estimates of $F^{\alpha,0}(t,v)$}

Next, we use Lemma \ref{dvkrho} for all $k=1, ..., n+1$ to find
\begin{equation}
\label{Dkrho}
\left\| \nabla^k_x \rho(t)\right\|_\infty \lesssim t^{-3-k}\mcG_v^k(t) \lesssim t^{-3-k}.
\end{equation}
We estimate the highest-order derivatives of $F^{\alpha,0}$ by invoking Lemma \ref{dvkFalpha0} with $k = n+1$ and using the preliminary rates from Step $1$.
In particular, separating the sum involving field derivatives into three terms, corresponding to $j=0$, the $j=1, .., n$ terms, and $j=n+1$, and using \eqref{IH}, \eqref{Dkrho}, and Lemma \ref{Dng}, we find
\begin{align*}
\infnorm{\nabla_v^{n+1} F^{\alpha,0}(t) - \nabla_v^{n+1}  F^{\alpha,0}_\infty }
& \lesssim \int_t^\infty  \sum_{j=0}^{n+1} s^{n+1-j}\biggl [s \infnorm{\nabla_x^{n+1-j} \rho(s)}\mcG_v^j(s) + \infnorm{\nabla_x^{n+1-j}E(s)}\mcG_v^{j+1}(s) \biggr ]ds\\
&\lesssim \int_t^\infty  \biggl (\sum_{j=0}^{n+1} s^{n+2-j} \infnorm{\nabla_x^{n+1-j} \rho(s)}\mcG_v^j(s) + s^{n+1}\infnorm{\nabla_x^{n+1}E(s)}\\
& \quad +  \sum_{j=1}^{n} s^{n+1-j} \infnorm{\nabla_x^{n+1-j}E(s)}\mcG_v^{j+1}(s) + \infnorm{E(s)}\mcG_v^{n+2}(s) \biggr ) ds\\
&\lesssim \int_t^\infty \biggl ( \sum_{j=0}^{n+1} s^{n+2-j} s^{-4-n+j} + s^{n+1} s^{-n-4} \ln(s)\\
& \quad  + \sum_{j=1}^{n} s^{n+1-j} s^{-n-3} + s^{-n-3} \ln^2(s) \biggr )ds\\
&\lesssim  \int_t^\infty s^{-2} \ ds \lesssim t^{-1}.
\end{align*}
Hence, the function
$$F^{\alpha,0}(t,v) = \int g^\alpha(t,y,v) \ dy$$
satisfies
\begin{equation}
\label{Dvnp1F0}
\left\| \nabla_v^{n+1} F^{\alpha,0}(t) - \nabla_v^{n+1}F^{\alpha,0}_\infty \right\|_\infty \lesssim t^{-1}.
\end{equation}
for any $\alpha = 1, ..., N$,
which establishes the existence and regularity of the limiting function $\nabla^{n+1}_v F^{\alpha,0}_\infty(v)$, as mentioned within Remark \ref{Finf}.

\vspace{0.1in}

\emph{\underline{Step 3:} Refined estimates of highest order derivatives}

With the convergence result for $\nabla_v^{n+1} F^{\alpha,0}(t,v)$ in place, we use Lemma \ref{LDEgeneral} with $k=n+2$ and $\ell = 0$ along with the preliminary estimate of $\mcG_v^{n+2}(t)$ in Lemma \ref{Dng} to conclude
\thinmuskip=1mu \medmuskip=2mu \thickmuskip=3mu
$$\sup_{x \in \mathbb{R}^3} \left | t^{n+4} \nabla_x^{n+2}E(t,x) - \nabla_v^{n+2} E_{0,\infty} \left( \frac{x}{t} \right) \right | \lesssim \ln(t) \biggl ( \left\Vert \nabla_v^{n+1} F^{\alpha, 0}(t) -  \nabla_v^{n+1}F^{\alpha, 0}_\infty(t) \right\Vert_\infty +  t^{-1} \mcG_v^{n+2}(t) \biggr ) \lesssim t^{-1}\ln^3(t),$$
\thinmuskip=3mu \medmuskip=4mu \thickmuskip=5mu
which, because the limiting function satisfies $E_{0,\infty} \equiv 0$, further implies a refined estimate of field derivatives of order $n+2$, namely
$$\| \nabla_x^{n+2}E(t) \|_\infty \lesssim t^{-n-5}\ln^3(t).$$
By revisiting \eqref{DngDnE}, this faster decay then implies the uniform bound $\mcG_v^{n+2}(t) \lesssim 1$, and thus
\begin{equation}
\label{Gxvnp2}
\mcG_{x,v}^{n+2}(t) + \mcG_v^{n+2}(t) \lesssim 1
\end{equation}
from Lemma \ref{Dng}. The bound on highest-order field derivatives is then immediately refined by \eqref{Gxvnp2} to give
\begin{equation}
\label{tnp4Dxnp2E}
\sup_{x \in \mathbb{R}^3} \left | t^{n+4} \nabla_x^{n+2}E(t,x) - \nabla_v^{n+2} E_{0,\infty} \left( \frac{x}{t} \right) \right | \lesssim t^{-1}\ln(t).
\end{equation}
%which yields
%\begin{equation}
%\label{Dxnp2E}
%\| \nabla_x^{n+2}E(t) \|_\infty \lesssim t^{-n-5}\ln(t).
%\end{equation}

\vspace{0.1in}

\emph{\underline{Step 4:} Improved estimates of lower order field derivatives}

The goal of this step is to improve the decay rates for lower order field derivatives and prove that for any $k\in\{1,\hdots,n+1\}$, we have $\infnorm{\nabla^k_x E(t)} \lesssim t^{-n-4}.$
We first estimate field derivatives of order $n+1$ and remove the logarithm in the estimate appearing within the induction hypothesis \eqref{IH}. In particular, we use Lemma \ref{DEgeneral} with $k =n+1$ and $\ell = 0$, which implies 
$$\sup_{x \in \mathbb{R}^3} \left | t^{n+3} \nabla_x^{n+1}E(t,x)  - \nabla_v^{n+1} E_{0,\infty} \left( \frac{x}{t}\right)  \right | \lesssim \max_{\alpha = 1, ..., N} \Vert \nabla_v^{n+1}F^{\alpha,0}(t) - \nabla_v^{n+1}F^{\alpha, 0}_\infty\Vert_\infty +  t^{-1} \mcG_v^{n+2}(t) $$
so that with \eqref{Dvnp1F0} and \eqref{Gxvnp2}, we arrive at
\begin{equation}
\label{tnp3Dxnp1E}
\sup_{x \in \mathbb{R}^3} \left | t^{n+3} \nabla_x^{n+1}E(t,x) - \nabla_v^{n+1} E_{0,\infty} \left( \frac{x}{t} \right) \right | \lesssim t^{-1}.
\end{equation}
and thus
\begin{equation}
\label{Dxnp1E}
\| \nabla_x^{n+1}E(t)\|_\infty \lesssim t^{-n-4}.
\end{equation}

Next, we will improve the estimate of $k$th order field derivatives for $k=1, ..., n$.
To do so, we will utilize Lemma \ref{DEgeneral}, but we must first bound the terms appearing on the right side of the associated inequality.
To begin, we invoke Lemma \ref{dvkFell} for any $k=1,\hdots,n$ and $\ell=0,\hdots,n-k$ and use \eqref{IH} to find
\begin{align*}
\infnorm{\nabla_v^k F^{\alpha,\ell}(t)-\nabla_v^k F^{\alpha,\ell}_\infty} 
	& \lesssim \int_t^\infty \sum_{j=0}^{k+\ell} s^{k+\ell-j}\infnorm{D_x^{k+\ell-j}E(s)} \Big( s \mcG_{x,v}^{j+1}(s) + \mcG_v^{j+1}(s) \Big) ds\\
    &\lesssim \int_t^\infty  \sum_{j=0}^{k+\ell} s^{k+\ell-j} \cdot s^{-n-3} \cdot s \ ds\\
	&\lesssim \int_t^\infty  \sum_{j=0}^{k+\ell} s^{k+\ell-n-j-2} ds \lesssim t^{k+\ell-n-1}
\end{align*}
as $k + \ell \leq n$. 
As before, this provides the existence and regularity of the limiting functions $\nabla^{k}_v F^{\alpha,\ell}_\infty(v)$
for any $\alpha = 1, ..., N$, $k=1,\hdots,n$, and $\ell=0,\hdots,n-k$, as mentioned within Remark \ref{Finf}.
With this, we invoke Lemma \ref{Drhogeneral} for any $k=1,\hdots,n$ and $\ell=0,\hdots,n-k$ to find
$$\sup_{x \in \bfR^3} \left | t^k \nabla^k_x\rho_\ell(t,x) - \nabla^k_v\rho_{\ell,\infty}\left(\frac{x}{t} \right) \right | \lesssim \max_{\alpha = 1, ..., N} \| \nabla^k_v F^{\alpha, \ell}(t)  - \nabla^k_v F^{\alpha, \ell}_\infty \|_\infty \lesssim t^{k+\ell-n-1}.$$
For any $\ell\in\{0,\hdots,n-1\},$ we have $\rho_{\ell,\infty}\equiv 0$, and hence this estimate gives
$$\infnorm{\nabla_x^k \rho_{\ell}(t) }\lesssim t^{\ell -n-1}$$
for any $k=1,\hdots,n$ and $\ell=0,\hdots,n-k$.

Similarly, using Lemma \ref{dvkFell} for any $k=1,\hdots,n$ and $\ell=n+1-k$ with \eqref{IH}, \eqref{Gxvnp2}, and \eqref{Dxnp1E} gives
\begin{align*}
	\infnorm{\nabla_v^k F^{\alpha,n+1-k}(t)-\nabla_v^k F^{\alpha,n+1-k}_\infty} &\lesssim \int_t^\infty \sum_{j=0}^{n+1} s^{n+1-j}\infnorm{\nabla_x^{n+1-j}E(s)} \Big( s \mcG_{x,v}^{j+1}(s) + \mcG_v^{j+1}(s) \Big)\ ds\\
	&\lesssim \int_t^\infty \Bigg[ s^{n+1}\infnorm{\nabla_x^{n+1}E(s)} ( s +1) + \infnorm{E(s)} \Big( s \mcG_{x,v}^{n+2}(s) + \mcG_v^{n+2}(s) \Big) \\
	&\qquad\qquad+\sum_{j=1}^{n} s^{n+1-j}\infnorm{\nabla_x^{n+1-j}E(s)} \Big( s \mcG_{x,v}^{j+1}(s) + \mcG_v^{j+1}(s) \Big) \Bigg]\ ds\\
	&\lesssim \int_t^\infty \sqrP{s^{n+2}s^{-n-4} + s^{-n-2}+\sum_{j=1}^{n}s^{n+2-j}s^{-n-3}} ds \lesssim t^{-1}.
\end{align*}
Combining these estimates and invoking Lemma \ref{DEgeneral} for $k = 1, ..., n$, and $\ell=n+1-k$ with \eqref{Gxvnp2}, we have 
\begin{align*}
	\sup_{x \in \bfR^3} \left | t^{n+3} \nabla^k_xE(t,x) - \nabla^k_vE_{n+1-k,\infty}\left(\frac{x}{t} \right) \right | 
	& \lesssim \sum_{m=0}^{n-k} t^{n+1-k -m} \| \nabla_x^k \rho_m(t) \|_\infty\\
	& \quad \quad + \max_{\alpha = 1, ..., N} \left\| \nabla^k_v F^{\alpha,n+1-k}(t)  - \nabla^k_v F^{\alpha, n+1-k}_\infty \right\|_\infty + t^{-1} \mcG_v^{n+2}(t)\\
	&\lesssim  \sum_{m=0}^{n-k} t^{n+1-k -m} \cdot t^{m-n-1}+t^{-1} \lesssim t^{-k} + t^{-1} \lesssim t^{-1}
\end{align*}
as $k \geq 1$. % we find
% $$\sup_{x \in \bfR^3} \left | t^{n+3} \nabla^k_xE(t,x) - \nabla^k_vE_{n+1-k,\infty}\left(\frac{x}{t} \right) \right |  \lesssim t^{-1}$$
% for all $k = 1, ..., n$.
Assembling this  with \eqref{tnp3Dxnp1E} results in
\begin{equation}
\label{tnp3DkxE}
\sup_{x \in \bfR^3} \left | t^{n+3} \nabla^k_xE(t,x) - \nabla^k_vE_{n+1-k,\infty}\left(\frac{x}{t} \right) \right |  \lesssim t^{-1} 
\end{equation}
for all $k = 1, ..., n+1$.
Therefore, as $n+1-k \leq n$ for $k \geq 1$, we deduce that $E_{n+1-k,\infty}\equiv 0$ for all $k=1,\hdots,n+1$, which further implies for all $k=1,\hdots,n+1$
\begin{equation}
\label{DkxE}
\infnorm{\nabla^k_x E(t)} \lesssim t^{-n-4}.
\end{equation}

\vspace{0.1in}

\emph{\underline{Step 5:} Establish next order limiting distribution}

With the improved field derivative rates, we now establish the convergence of the next order distribution to its limit.
Invoking Lemma \ref{dvkFell} with $k = 0$ and $\ell = n+1$ and using \eqref{IH}, \eqref{Gxvnp2}, and \eqref{DkxE}, we find
\begin{align*}
\infnorm{F^{\alpha,n+1}(t) - F^{\alpha,n+1}_\infty} 
& \lesssim \int_t^\infty  \sum_{j=0}^{n+1} s^{n+1-j}\infnorm{\nabla_x^{n+1-j}E(s)} \Big( s \mcG_{x,v}^{j+1}(s) + \mcG_v^{j+1}(s) \Big) \ ds \\
& = \int_t^\infty \biggl [\sum_{j=0}^{n} s^{n+1-j}\infnorm{\nabla_x^{n+1-j}E(s)} \Big( s \mcG_{x,v}^{j+1}(s) + \mcG_v^{j+1}(s) \Big)\\
& \qquad  + \Vert E(s) \Vert_\infty \Big( s \mcG_{x,v}^{n+2}(s) + \mcG_v^{n+2}(s) \Big) \biggr ] \ ds \\
& \lesssim \int_t^\infty \left(\sum_{j=0}^{n} s^{n+2-j} \cdot s^{-n-4}  + s^{-n-2} \right) \ ds \\
& \lesssim \int_t^\infty \sum_{j=0}^{n} s^{-2-j} \ ds  \lesssim \int_t^\infty s^{-2} \ ds \  \lesssim t^{-1}.
\end{align*}
Hence, the function $F^{\alpha,n+1}(t,v)$ satisfies
\begin{equation} 
\label{eq:Fn+1}
	\| F^{\alpha,n+1}(t) - F^{\alpha,n+1}_\infty \|_\infty \lesssim t^{-1}
\end{equation}
for any $\alpha = 1, ..., N$,
which establishes the existence of the limiting function $F^{\alpha,n+1}_\infty(v)$, as mentioned within Remark \ref{Finf}.
Using this, we define the next order limiting charge density and electric field via
$$\rho_{n+1,\infty}(v) = \sum_{\alpha = 1}^N q_\alpha F_\infty^{\alpha, n+1}(v)
\qquad \mathrm{and} \qquad
E_{n+1,\infty}(v) =\nabla_v (\Delta_v)^{-1} \rho_{n+1,\infty}(v).$$

\vspace{0.1in}

\emph{\underline{Step 6:} Convergence of density and field to next order distribution}

To establish the convergence of the next order density, we first use Lemma \ref{LDensitygeneral} with $\ell=n+1$ to find
	\[\sup_{x\in\R^3}\left|{t^{n+4}\rho(t,x)-\rho_{n+1,\infty}\rndP{\frac{x}{t}}}\right| \lesssim \sum_{j=0}^{n}t^{n+1-j}\infnorm{\rho_j(t)}+\max_{\alpha=1,\hdots,N}\infnorm{F^{\alpha,n+1}(t)-F_\infty^{\alpha,n+1}} + t^{-1}\mcG_v^{n+2}(t).\]
	The second term is easily estimated via \eqref{eq:Fn+1} and the third term via \eqref{Gxvnp2} so that each is $\mathcal{O} \left(t^{-1} \right)$. Hence, we focus on estimating the first term on the right side of the inequality.
	
	First, using Lemma \ref{dvkFell} with $k=0$ and $\ell=1,\hdots,n$, as well as \eqref{IH} and \eqref{DkxE}, we find
	\begin{align*}
		\infnorm{F^{\alpha,\ell}(t)-F_\infty^{\alpha,\ell}} 
		&\lesssim \int_t^\infty \sum_{j=0}^{\ell} s^{\ell-j}\infnorm{\nabla_x^{\ell-j}E(s)}\Big(s\mcG_{x,v}^{j+1}(s)+\mcG_v^{j+1}(s)\Big) \ ds \\
		&\lesssim \int_t^\infty \left(\sum_{j=0}^{\ell-1} s^{\ell-j+1}\infnorm{\nabla_x^{\ell-j}E(s)} + s \Vert E(s) \Vert_\infty \right) ds \\
		&\lesssim \int_t^\infty \left(\sum_{j=0}^{\ell-1} s^{\ell-j-n-3} + s^{-n-2} \right) ds \lesssim t^{\ell-n-2}+t^{-n-1}\lesssim t^{\ell-n-2}
	\end{align*}
	for $\ell=1,\hdots,n$.
	For $\ell=0$, we use Lemma \ref{dvkFalpha0} with $k=0$ and \eqref{IH} to arrive at
	\[\infnorm{F^{\alpha,0}(t)-F^{\alpha,0}_\infty} \lesssim \int_t^\infty \Big( s\infnorm{\rho(s)}+\infnorm{E(s)} \Big) \ ds \lesssim \int_t^\infty s^{-n-3}ds \lesssim t^{-n-2},\]
	which provides the same rate as above when $\ell = 0$. 
	Putting these together and using Lemma \ref{Drhogeneral} with $k=0$ and $\ell = 0,\hdots,n$ yields
	\begin{equation*} 
	\label{eq:rhoell}
	\sup_{x \in \bfR^3} \left|{\rho_\ell(t,x)-\rho_{\ell,\infty}\rndP{\frac{x}{t}}}\right|\lesssim \max_{\alpha=1,\hdots,N}\infnorm{F^{\alpha,\ell}(t)-F_\infty^{\alpha,\ell}} \lesssim t^{\ell-n-2}
	\end{equation*}
	for any $\ell = 0,\hdots,n$.
	As $\rho_{\ell,\infty}\equiv 0$ for $\ell = 0,\hdots,n$, we find
	\[ \infnorm{\rho_\ell(t)} \lesssim t^{\ell-n-2} .\]
	Inserting this within the sum then gives 	
	\begin{equation}
	\label{rhosum}
	\sum_{j=0}^{n}t^{n+1-j}\infnorm{\rho_j(t)} \lesssim \sum_{j=1}^{n}t^{n+1-j}t^{j-n-2} \lesssim t^{-1},
	\end{equation}
	and adding this to the previously established estimates for the other terms, we find
	\begin{equation}
	\label{rhonp1}
	\sup_{x\in\R^3}\left|{t^{n+4}\rho(t,x)-\rho_{n+1,\infty}\rndP{\frac{x}{t}}}\right| \lesssim t^{-1}.
	\end{equation}

Turning to the next order field approximation, we first apply Lemma \ref{DEgeneral} with $k=0$ and $\ell=n+1$ to find 
	\[\sup_{x\in\R^3}\left|{t^{n+3} E(t,x)-E_{n+1,\infty}\rndP{\frac{x}{t}}}\right|\lesssim \sum_{j=0}^{n}t^{n+1-j}\infnorm{\rho_j(t)}+ \max_{\alpha=1,\hdots,N}\Vert F^{\alpha,n+1}(t)-F_\infty^{\alpha,n+1} \Vert_\infty  + t^{-1}\mcG_v^{n+2}(t).\]
	%Finally,
	As before, the third term is $\mathcal{O} \left(t^{-1} \right)$ due to \eqref{Gxvnp2}, while the second term is $\mathcal{O} \left(t^{-1} \right)$  due to \eqref{eq:Fn+1}.
	To estimate the first term, we again use \eqref{rhosum} so that
%	Lemma \ref{LE} and \eqref{eq:rhoell} so that 
%	\begin{align*}
%	\left|{E_\ell(t,x)-E_{\ell,\infty}\rndP{\frac{x}{t}}}\right| &\leq \Vert \nabla_x{\rndP{\Delta_x}}^{-1}\rndP{\rho_\ell(t) -\rho_{\ell,\infty} } \Vert_\infty \\
%	&\lesssim \infnorm{\rho_\ell(t)-\rho_{\ell,\infty}}\\
%	&\lesssim t^{\ell-n-2}
%	\end{align*}
%	for any $\ell = 0,\hdots,n$.
%	Noting that $E_{\ell,\infty}\equiv 0$ for $\ell = 0,\hdots,n$ provides
%	\[ \infnorm{E_\ell(t)} \lesssim t^{\ell-n-2}.\]
%	%
%	Inserting this within the first term then gives 	
%	\begin{align*}
%	\sum_{j=0}^{n}t^{n+1-j}\infnorm{E_j(t)} \lesssim \sum_{j=1}^{n}t^{n+1-j}t^{j-n-2} \lesssim t^{-1},
%	\end{align*}
%	and adding this to the previously established estimates, we find
putting these estimates together yields
	\begin{equation}
	\label{Enp1}
	\sup_{x\in\R^3}\left|{t^{n+3}E(t,x)-E_{n+1,\infty}\rndP{\frac{x}{t}}}\right| \lesssim t^{-1}.
	\end{equation}

Finally, assembling the estimates \eqref{Gxvnp2}, \eqref{tnp4Dxnp2E}, \eqref{tnp3DkxE}, \eqref{rhonp1}, and \eqref{Enp1}, as well as the previous bounds on $\mcG^k_{x,v}(t)$ and $\mcG^k_{v}(t)$ for $k = 0, ..., n+1$ implied by the induction hypothesis \eqref{IH}, then completes the inductive step as $Q_n$ has been established.
Hence, the stated decay rates of the charge density, electric field, field derivatives, and derivatives of $g^\alpha$  hold for any $n \in \mathbb{N}$.

Lastly, to justify the rate of convergence of $g^\alpha(t,x,v)$ and its derivatives, we note that using the stated decay rate of the field and the uniform boundedness of derivatives of $g^\alpha$, the Vlasov equation \eqref{VPg} yields
$$\Vert \partial_t g^\alpha (t) \Vert_\infty \lesssim t\Vert E(t) \Vert_\infty \mcG_{x,v}^1(t) + \Vert E(t) \Vert_\infty \mcG_v^1(t) \lesssim t^{-n-2}.$$
Upon integrating in $t$, we find for any $\alpha = 1, ..., N$
$$\sup_{(x,v) \in \bfR^6} \left | g^\alpha \left(t,x, v \right) - f^\alpha_\infty(x,v) \right | \lesssim \int_t^\infty  \Vert \partial_t g^\alpha (s) \Vert_\infty \ ds  \lesssim t^{-n-1}.$$
Similarly, taking any $k$th-order $x$-derivative and $\ell$th-order $v$-derivative in the Vlasov equation with $k,\ell \in \mathbb{N}$ and $k + \ell \leq n+1$ and taking the supremum over $(x,v)\in \mathbb{R}^6$, we find
	\begin{align*}
	\Vert \partial_t D^k_x D^\ell_v g^\alpha(t) \Vert_\infty 
	 \lesssim& \sum_{i=0}^{k-1} \sum_{j=0}^\ell t^j \| \nabla_x^{i+j} E(t) \Vert_\infty \biggl ( t \mcG_{x,v}^{k+\ell-i-j+1}(t) + \mcG_{x,v}^{k+\ell-i-j+1}(t) \biggr )\\
	 &+ \sum_{j=0}^\ell t^j \| \nabla_x^{j+k} E(t) \Vert_\infty \biggl ( t \mcG_{x,v}^{\ell-j+1}(t) + \mcG_{v}^{\ell-j+1}(t) \biggr )
	\end{align*}
where the indices $i$ and $j$ represent the number of $x$ and $v$ derivatives that are applied to the field, respectively. 
As this holds for arbitrary derivatives, we decompose the double sum and use the previous estimates on the field \eqref{Enp1}, field derivatives \eqref{DkxE}, and derivatives of the distribution function via \eqref{IH} and \eqref{Gxvnp2} to find 
	\begin{align*}
	\Vert \partial_t \nabla^k_x \nabla^\ell_v g^\alpha(t) \Vert_\infty 
	 \lesssim&
	  (t+1) \Vert E(t) \Vert_\infty \mcG_v^{k+\ell+1}(t) + \sum_{j=1}^\ell t^j (t+1) \| \nabla_x^{j} E(t) \Vert_\infty \mcG_{x,v}^{k+\ell-j+1}(t)\\
	  & \ + \sum_{i=1}^{k-1} \sum_{j=0}^\ell t^{j}(t+1) \| \nabla_x^{i+j} E(t) \Vert_\infty \mcG_{x,v}^{k+\ell-i-j+1}(t)  + \sum_{j=0}^\ell t^j(t+1) \| \nabla_x^{j+k} E(t) \Vert_\infty\\
	   \lesssim& t^{-n-2} + \sum_{j=1}^\ell t^{j-n-3} + \sum_{i=1}^{k-1} \sum_{j=0}^\ell t^{j-n-3} + \sum_{j=0}^\ell t^{j-n-3}\\
	    \lesssim&  \max\{t^{-n-2}, t^{\ell-n-3}\}.
\end{align*}
Upon integrating in $t$, we find for any $\alpha = 1, ..., N$
$$\sup_{(x,v) \in \bfR^6} \left | \nabla_x^k \nabla_v^\ell g^\alpha \left(t,x, v \right) - \nabla_x^k \nabla_v^\ell  f^\alpha_\infty(x,v) \right | \lesssim \int_t^\infty  \Vert \partial_t \nabla_x^k \nabla_v^\ell g^\alpha (s) \Vert_\infty \ ds  \lesssim \max\{t^{-n-1}, t^{\ell-n-2}\}.$$
for $k + \ell \leq n+1$.
The estimates are analogous for the boundary cases $k=0$, $1 \leq \ell \leq n+1$ and $\ell=0$, $ 1 \leq k \leq n+1$.
In particular, the uniform convergence for all derivatives of order $n+1$ and the previously known compact support of the limiting function then implies $f_\infty^\alpha \in C_c^{n+1}(\bfR^6)$, which provides the regularity needed to justify \eqref{eq:F-al-inf}, as stated in Remark \ref{Finf}.
This completes the proof of  Theorem \ref{T2}.
\end{proof}

Next, we use the results of Theorems \ref{oldT1} and \ref{oldT2} and invoke Theorem \ref{T2} to prove Theorem \ref{T1}.

\begin{proof}[Proof of Theorem \ref{T1}]
Assume \eqref{Assumption} holds and $\mcM = 0$. 
Then, for $m = 0$, we merely take $\rho_{0,\infty} \not\equiv 0$ and apply Theorem \ref{oldT1}, which yields the sharp bounds
$$
\begin{gathered}
\| \rho(t)\|_\infty \sim t^{-3},\\
\| E(t)\|_\infty \sim t^{-2},\\
\| \nabla_x E(t)\|_\infty \sim t^{-3}
\end{gathered}
$$
and the modified scattering result
\begin{equation}
\label{modscattering}
\sup_{(x,v) \in \bfR^6} \left | f^\alpha \left(t,x +vt - \frac{q_\alpha}{m_\alpha} \ln(t)E_{0,\infty}(v), v \right) - f^\alpha_\infty(x,v) \right |  \lesssim t^{-1}\ln^{4}(t)
\end{equation}
for every $\alpha = 1, ..., N$ and any associated solution of \eqref{VP}.

Otherwise, we let $m \geq 1$ be given, define $n = m - 1$, and take $\rho_{\ell,\infty} \equiv 0$ for all $\ell = 0, ..., n$ with $\rho_{n+1,\infty} \not\equiv 0$
Then, applying Theorem \ref{T2} with $n = m-1$ gives
$$
\begin{gathered}
\| \rho(t)\|_\infty \sim t^{-m-3},\\
\| E(t)\|_\infty \sim t^{-m-2},\\
\| \nabla_x^k E(t)\|_\infty \sim t^{-m-3} \\
\end{gathered}
$$
for every $k =1, ..., m$.
Additionally, the distribution function scatters linearly so that
\begin{equation}
\label{scattering}
\sup_{(x,v) \in \bfR^6} \left | f^\alpha(t,x +vt, v) - f^\alpha_\infty(x,v) \right |  \lesssim t^{-m}
\end{equation}
for every $\alpha = 1, ..., N$ and any associated solution of \eqref{VP}.

It remains to justify the existence of solutions $f^\alpha \in C^{m+1} \left((0,\infty) \times \mathbb{R}^6 \right)$ for $\alpha = 1, ..., N$ 
that satisfy $\rho_{0,\infty} \not\equiv 0$ or, alternatively, the conditions
$\rho_{\ell,\infty} \equiv 0$ for all $\ell = 0, ..., m-1$
and
$\rho_{m,\infty} \not\equiv 0$.
To do this, we merely construct smooth functions with compact support whose moments, up to a desired order, must vanish.
Thus, we build well-behaved limits and utilize the scattering map constructed within \cite{Flynn} (see also \cite{Schlue-Taylor}) to guarantee the existence of solutions that tend to these limits as $t \to \infty$.
We perform this separately for $m = 0$ and $m \in \mathbb{N}$.

First, for $m=0$, we let a nonzero function $\eta \in C_c^{1}(\mathbb{R}^3)$ be given with
$$\int \eta(v) \ dv  = 0.$$
Then, we choose
$$f^\alpha_{0,\infty}(x,v) = \phi_0(x) \psi_0^\alpha(v)$$
where $\phi_0 \in C_c^{1}(\mathbb{R}^3)$ is nonnegative (but nontrivial), and for every $\alpha = 1, ..., N$
the functions $\psi_0^\alpha \in C_c^{1}(\mathbb{R}^3)$ are nonnegative and
satisfy the constraint
$$\sum_{\alpha = 1}^N q_\alpha \psi_0^\alpha(v) = \eta(v).$$
With this, we have
	\begin{align*}
	\rho_{0,\infty}(v)
	=
	\int \sum_{\alpha = 1}^N q_\alpha f^\alpha_{0,\infty}(x,v) \ dx
	=
	\left( \int \phi_0(x) \ dx \right)\sum_{\alpha = 1}^N q_\alpha \psi^\alpha_{0}(v)
	=
	 \left( \int  \phi_0(x) \ dx \right) \eta(v) 
	 \not\equiv 0
	\end{align*}
and
\begin{align*}
	\int \rho_{0,\infty}(v) \ dv
	= \left( \int  \phi_0(x) \ dx \right) \left( \int \eta(v) \ dv \right)
	 = 0,
\end{align*}
thereby satisfying the neutrality condition $\mcM = 0$ imposed by \eqref{Pinfmass}. 

Next, for a given $m \in \N$, we take $p > m+2$ and define $\Phi \in C_c^{m+1}(\mathbb{R})$ to be the corresponding weighted Gegenbauer polynomial of degree $m$ (see \cite{AbramSteg}), namely
$$\Phi_m(x) = \left( 1 - x^2 \right)^{p - \frac{1}{2}} {C}^p_{m}(x) \mathbbm{1}_{[-1,1]}$$
where $C^p_k(x)$ is a $k$th order Gegenbauer polynomial satisfying the orthogonality relationship
$$\int_{-1}^1 \left( 1 - x^2 \right)^{p - \frac{1}{2}} C^p_{k}(x)  C^p_{\ell}(x) \ dx = 0$$
for all $k,\ell \in \N_0$ with $k \neq \ell$
and
$$\int_{-1}^1 \left( 1 - x^2 \right)^{p - \frac{1}{2}} \left [ C^p_{k}(x)  \right]^2 \ dx > 0$$
for any $k \in \N_0$.
Because each $C^p_{k}(x)$ is a polynomial of degree $k$, we may normalize    $\Phi_m(x)$ so that
$$ \int x^{m} \Phi_m(x) \ dx = 1,$$
and orthogonality implies
$$ \int x^k \Phi_m(x) \ dx = 0, \qquad \mathrm{for \ all \ } k = 0, ...,m-1.$$
We note that the support of $\Phi_m$ can be rescaled to be a compact set of arbitrary size, if necessary, while maintaining the moment and regularity properties.
%Such functions are known to exist and are often constructed via arbitrarily smooth orthogonal polynomials with compact support (e.g. weighted Gegenbauer polynomials, featured within Section 10 of \cite{Tadmor}).
Then, letting $\Psi \in C_c^{m+1}(\bfR)$ be any function satisfying
$$\int \Psi(x) \ dx = 1,$$
we define $\mu_m \in C_c^{m+1}(\mathbb{R}^3)$ by
$$\mu_m (x) =  \Phi_m(x_1) \Psi(x_2) \Psi(x_3). $$
%
%If $m=0$, we find from these properties
%$$\int \mu_0(x) \ dx  = 1.$$
%In contrast, i
Then, for all $\beta \in \mathbb{N}_0^3$ with $|\beta| \leq m$, this yields 
\begin{align*}	
\int x^{\beta} \mu_m(x) \ dx & =  \left( \int x_1^{\beta_1} \Phi_m(x_1) \ dx_1 \right) \left( \int x_2^{\beta_2} \Psi(x_2) \ dx_2 \right) \left( \int x_3^{\beta_3} \Psi(x_3) \ dx_3 \right)\\
& = 
\begin{cases}
1 & \mathrm{if} \ \beta = (m,0,0) \\
0 & \mathrm{else,}
\end{cases}
\end{align*}
as $\beta_1 \in \{0, ..., m-1\}$ implies that the first integral vanishes due to the orthogonality condition stated above.

With the $\mu_m(x)$ functions in place for every $m \in \mathbb{N}$, we can now define the remaining scattering limits.
For a given $m \in \N$ we choose
$$f^\alpha_{m,\infty}(x,v) = \phi_m^\alpha(x) \psi_m(v)$$
for every $\alpha = 1, ..., N$
where
$\psi_m \in C_c^{m+1}(\mathbb{R}^3)$ 
is nonnegative and satisfies
$$\frac{\partial^m\psi_m}{\partial v_1^m}(v) \not\equiv 0,$$
and the collection of nonnegative spatial distributions 
$\phi_m^\alpha \in C_c^{m+1}(\mathbb{R}^3)$
satisfies the constraint
$$\sum_{\alpha = 1}^N q_\alpha \phi_m^\alpha(x) = \mu_m(x).$$
This implies that for all $\ell = 0, ..., m-1$,
\begin{align*}
\rho_{\ell, \infty}(v) & =  \sum_{\alpha = 1}^N q_\alpha F^{\alpha,\ell}_\infty(v) = \sum_{\alpha = 1}^N \sum_{|\beta| = \ell} \frac{{(-1)^\beta}}{\beta!}\! \int D_v^\beta x^\beta q_\alpha f^\alpha_{m, \infty}(x,v) \ dx\\
& =  \sum_{|\beta| = \ell} \frac{{(-1)^\beta}}{\beta!} D_v^\beta \psi_m(v) \left(\int x^\beta \sum_{\alpha = 1}^N  q_\alpha \phi_m^\alpha(x) \ dx \right)\\
& =  \sum_{|\beta| = \ell} \frac{{(-1)^\beta}}{\beta!} D_v^\beta \psi_m(v) \left(\int x^\beta \mu_m(x) dx \right) \equiv 0
\end{align*}
as the lower-order moments of $\mu_m(x)$ must vanish.
Additionally, for the final limiting density we find
\begin{align*}
\rho_{m, \infty}(v) & =  \sum_{|\beta| = m} \frac{{(-1)^\beta}}{\beta!} D_v^\beta \psi_m(v) \left(\int x^\beta \mu_m(x) dx \right)\\
& =  \frac{{(-1)^m}}{m!}\frac{\partial^m\psi_m}{\partial v_1^m}(v) \left(\int x_1^m \mu_m(x) dx \right)\\
& =  \frac{{(-1)^m}}{m!}\frac{\partial^m\psi_m}{\partial v_1^m}(v) \not\equiv 0.
\end{align*}
Finally, as we have constructed a family of limits $f^\alpha_{m,\infty} \in C_c^{m+1}(\mathbb{R}^6) \subset W^{2,\infty}(\mathbb{R}^6)$ for every $\alpha = 1, ..., N$ and $m \in \N_0$, an application of  \cite[Theorem 1.1(ii) and Remark 1.2(5)]{Flynn} (see also \cite[Theorem 1.1]{Schlue-Taylor}) implies that for each $m \in \N_0$ there exists a unique smooth solution of \eqref{VP}, in this case $ f_m^\alpha \in C^{m+1} \left((0,\infty) \times \mathbb{R}^6 \right)$ for every $\alpha = 1, ..., N$ that is associated to this limit $f^\alpha_{m,\infty}(x,v)$ via \eqref{modscattering} for $m=0$ and \eqref{scattering} for $m \geq 1$, respectively, and the proof is complete.
\end{proof}

\section{Derivatives of Translated Distributions}
\label{sec:bounds-derivs-g}
This section is dedicated to the proofs of Lemmas \ref{Dng0} and \ref{Dng}.
%
%\begin{lemma}
%\label{Dng}
%Let $n \in \N_0$ be given. 
%For $n=0$, assume \eqref{BC}, while for $n \geq 1$, assume \eqref{IH}.
%Then, we have
%%\begin{equation}
%%\label{DngDnE}
%%\mcG_v^{n+2}(t) \lesssim 1 + \int_1^t \biggl[ s^{n+3} \|D_x^{n+2}E(s)\|_{\infty} + s \Vert D_xE(s) \Vert_\infty \left( s  \mcG_{x,v}^{n+2}(s)  + \mcG_v^{n+2}(s) \right)\biggr ] \ ds,
%%\end{equation}
%\begin{equation}
%\label{DngDnE}
%\mcG_v^{n+2}(t) \lesssim 1 + \int_1^t s^{n+3} \|\nabla_x^{n+2}E(s)\|_{\infty} \ ds,
%\end{equation}
%%
%and, as a preliminary estimate,
%$$\Vert \nabla_x^{n+2} E(t) \Vert_\infty \lesssim t^{-4-n} \ln(t), \qquad \mcG_{x,v}^{n+2}(t) \lesssim 1, \qquad \mathrm{and} \qquad  \mcG_v^{n+2}(t) \lesssim \ln^2(t).$$
%\end{lemma}
Before proving these lemmas, we must define the quantity $\mcD_k(t)$ which includes a total of $k$ derivatives (both in $x$ and in $v$) of $g$ with a weight $t^{-j}$ depending on how many $v$ derivatives appear, namely
	\begin{align*}
	\mcD_k(t) :=1+ \sum_{j=0}^kt^{-j} \max_{\alpha = 1, ..., N} \sup_{\substack{|\beta_x|+|\beta_v|=k\\|\beta_v|=j}} \|D_v^{\beta_v} D_x^{\beta_x}g^\alpha(t)\|_\infty.
	\end{align*}
For brevity, we denote
	\[
	\mcD_k^j(t)
	:=
	\max_{\alpha = 1, ..., N} \sup_{\substack{|\beta_x|+|\beta_v|=k\\|\beta_v|=j}} \|D_v^{\beta_v} D_x^{\beta_x}g^\alpha(t)\|_\infty
	\]
for $j = 0, ..., k$ so that
	\[
	\mcD_k(t) =1+ \sum_{j=0}^kt^{-j}\mcD_k^j(t).
	\]

\begin{proof}[Proof of Lemmas \ref{Dng0} and \ref{Dng}]
%We prove the result for $n \geq 1$ and note that the decay estimates of field derivatives are actually one power stronger in the case $n=0$, so that any appearances of $s^{-1-n}$ terms below will be replaced by $s^{-2}$ whenever $n=0$. As these terms are integrable, the same techniques will also suffice in this case, as well.
As the proof of Lemma \ref{Dng0} follows similarly to that of Lemma \ref{Dng}, but in the case $n=0$, we combine their respective proofs here. 
The strategy of the proof proceeds as follows. To understand the asymptotic properties of $D_v^{\beta_v} D_x^{\beta_x} g^\alpha$ with $|\beta_x+\beta_v| = n+2$, we apply these derivatives to the Vlasov equation $\mcV_g g^\alpha = 0$. The idea is to get an expression that one can bound, and then integrate along characteristics in order to eliminate $\mcV_g$.  This will allow us to construct a bound on the sum of all $n+2$ order derivatives in $x$ and $v$ of $g^\alpha$ to show $\mcD_{n+2}(t)\lesssim 1$ following an application of Gronwall's inequality. 
The uniform bound on $\mcD_{n+2}(t)$ will then imply $\mcD_{n+2}^\ell(t) \lesssim t^\ell$ for every $\ell=0,\hdots,n+2$, which then refines the bound on $\infnorm{\nabla_x^{n+2}E(t)}$ stated in Lemma \ref{LDE}. The polynomial growth of each $\mcD_{n+2}^\ell(t)$ will then be iteratively improved by applying a Gronwall argument until each of these quantities is uniformly bounded for all $\ell\in\{1,\hdots,n\}$ and by $\ln(t)$ when $\ell=n+1$. In the final part of the proof, a logarithmic bound on $\mcD_{n+2}^{n+2}(t)$ will be established and the upper bound on $\mcD_{n+2}^{n+1}(t)$ will be improved to a constant, enabling us to  conclude  $\mcG_{x,v}^{n+2}(t)\lesssim 1$ and $\mcG_{v}^{n+2}(t)\lesssim \ln^2(t)$ for all $n\geq 0$ from \eqref{eq:G^k_v} and \eqref{eq:G^k_x,v}.
Throughout, we take $t \geq 1$ for simplicity, and it will always be specified whether $n=0$ or $n\geq 1$, as the base case (Lemma \ref{Dng0}) requires slightly different assumptions.

\vspace{0.1in}

\emph{\underline{Step 1:} Uniform bound on $\mcD_{n+2}(t)$}

 Taking derivatives in \eqref{VPg}, a simple computation leads us to a formula for the commutator, namely
	\begin{equation}
	\label{eq:k-derivs-g}
	\mcV_g\big(D_v^{\beta_v} D_x^{\beta_x} g^\alpha\big)
	=
	\sum_{j=0}^{n+1}\left(\!\begin{array}{c}n+2\\j\end{array}\!\right) \sum_{\substack{|\gamma_x+\gamma_v|=j \\ \gamma_x \preceq\beta_x, \gamma_v \preceq\beta_v}}(D_v^{\beta_v-\gamma_v} D_x^{\beta_x-\gamma_x}\mcV_g) \left( D_v^{\gamma_v} D_x^{\gamma_x}g^\alpha \right).
	\end{equation}
Here, the $\gamma_x$ and $\gamma_v$ derivatives represent those from the $\beta_x$ and $\beta_v$ derivative products which are not applied to the Vlasov operator, but instead act on $g^\alpha$.
Note that the term with all $n+2$ derivatives applied to $g^\alpha$ vanishes, as $D_v^{\gamma_v} D_x^{\gamma_x} \left( \mcV_g g^\alpha \right) = 0$ by the Vlasov equation.

We focus on the individual terms on the right hand side of \eqref{eq:k-derivs-g}. Recalling that the operator $\mcV_g$ is given by the expression 
	\[
	\mcV_g 
	=
	\partial_{t} + \frac{q_\alpha}{m_\alpha} E(t,x+vt) \cdot \left( - t \nabla_{x} + \nabla_{v} \right)
	\]
we find
	\begin{align*}
	D_v^{\beta_v-\gamma_v} D_x^{\beta_x-\gamma_x}\mcV_g
	=& \frac{q_\alpha}{m_\alpha}t^{|\beta_v-\gamma_v|}\left(D_x^{\beta_v+\beta_x-\gamma_v-\gamma_x}E\right)(t,x+vt)\cdot \left( - t\nabla_x + \nabla_v \right)
	\end{align*}	
so that
	\begin{align*}
	(D_v^{\beta_v-\gamma_v} D_x^{\beta_x-\gamma_x}\mcV_g) \left(D_v^{\gamma_v} D_x^{\gamma_x}g^\alpha \right)
	=&
	\frac{q_\alpha}{m_\alpha}t^{|\beta_v-\gamma_v|}\left(D_x^{\beta_v+\beta_x-\gamma_v-\gamma_x}E\right)(t,x+vt)\cdot \left( -t\nabla_x + \nabla_v \right) \left(D_v^{\gamma_v} D_x^{\gamma_x} g^\alpha \right).
	\end{align*}
Inserting this expression into \eqref{eq:k-derivs-g} we finally arrive at
$$ \mcV_g(D_v^{\beta_v} D_x^{\beta_x} g^\alpha)
	=
	\frac{q_\alpha}{m_\alpha}\sum_{j=0}^{n+1}\left(\hspace{-0.15cm}\begin{array}{c}n+2\\j\end{array}\hspace{-0.1cm}\right)\hspace{-0.2cm}
	\sum_{\substack{|\gamma_x+\gamma_v|=j \\ \gamma_x \preceq\beta_x, \gamma_v \preceq\beta_v}}\hspace{-0.4cm}
	t^{|\beta_v-\gamma_v|}\left(D_x^{\beta_v+\beta_x-\gamma_v-\gamma_x}E\right)(t,x+vt)\cdot \left( -t\nabla_x + \nabla_v \right) \left(D_v^{\gamma_v} D_x^{\gamma_x}g^\alpha \right).
$$
We then integrate this equation along characteristics, thus eliminating $\mcV_g$ from the left side, and take the supremum over $(x,v) \in \bfR^6$ to find
	\begin{align*}
		\|D_v^{\beta_v} D_x^{\beta_x} g^\alpha(t)\|_\infty
		 \lesssim
		1 + \int_1^t\sum_{j=0}^{n+1}\sum_{\substack{|\gamma_x+\gamma_v| = j \\ \gamma_x \preceq\beta_x, \gamma_v \preceq\beta_v}}
		& \bigg (s^{1+|\beta_v-\gamma_v|}\| D_x^{\beta_v+\beta_x-\gamma_v-\gamma_x}E(s)\|_\infty  \| \nabla_xD_v^{\gamma_v} D_x^{\gamma_x}g^\alpha(s)\|_\infty \\
		&  + s^{|\beta_v-\gamma_v|} \| D_x^{\beta_v+\beta_x-\gamma_v-\gamma_x}E(s)\|_\infty \| \nabla_vD_v^{\gamma_v} D_x^{\gamma_x}g^\alpha(s) \|_\infty \bigg ) \ ds,
	\end{align*}
	which, as $|\beta_x + \beta_v | = n+2$, implies
\begin{equation}
	\label{DngOrig}	
	\|D_v^{\beta_v} D_x^{\beta_x} g^\alpha(t)\|_\infty
	 \lesssim
	1+ \int_1^t\sum_{j=0}^{n+1}\sum_{\substack{|\gamma_x+\gamma_v| = j \\ \gamma_x \preceq\beta_x, \gamma_v \preceq\beta_v}}
	s^{|\beta_v-\gamma_v|} \|\nabla_x^{n+2-j}E(s)\|_{\infty} \left( s\mcD_{j+1}^{|\gamma_v|}(s) + \mcD_{j+1}^{|\gamma_v|+1}(s) \right)ds.
\end{equation}

Next, we separate the terms within the sum over $j$ on the right side of this inequality into $I + II + III + IV$, as follows. The $j=0$ term satisfies
$$I \lesssim s^{|\beta_v|} \|\nabla_x^{n+2}E(s)\|_{\infty} \left(s \mcG_{x,v}^1(s) + \mcG_{v}^1(s)\right),$$
the $j=1$ term is
$$II = \sum_{\substack{|\gamma_x+\gamma_v| = 1 \\ \gamma_x \preceq\beta_x, \gamma_v \preceq\beta_v}} 
s^{|\beta_v - \gamma_v|} \|\nabla_x^{n+1}E(s)\|_{\infty} \left( s\mcD_{2}^{|\gamma_v|}(s)
	+\mcD_{2}^{|\gamma_v|+1}(s) \right),$$
the remaining middle terms, for $n\geq 2$, are grouped together into
$$III = \sum_{j=2}^{n}\sum_{\substack{|\gamma_x+\gamma_v| = j \\ \gamma_x \preceq\beta_x, \gamma_v \preceq\beta_v}}
	s^{|\beta_v-\gamma_v|} \|\nabla_x^{n+2-j}E(s)\|_{\infty} \left( s\mcD_{j+1}^{|\gamma_v|}(s)
	+\mcD_{j+1}^{|\gamma_v|+1}(s) \right),$$
and finally the $j=n+1$ term is
$$IV = \sum_{\substack{|\gamma_x+\gamma_v| = n+1 \\ \gamma_x \preceq\beta_x, \gamma_v \preceq\beta_v}}
	s^{|\beta_v-\gamma_v|} \|\nabla_x E(s)\|_{\infty} \left( s\mcD_{n+2}^{|\gamma_v|}(s)
	+ \mcD_{n+2}^{|\gamma_v|+1}(s) \right).$$
For $n=0$, we note that only the terms $I$ and $IV$ (which is equivalent to $II$ in this case) appear, while for $n=1$ only the terms $I$, $II$, and $IV$ appear.

Now, applying the previously established decay estimates stemming from the vanishing limit(s) of the charge density in each lemma, we bound the terms in this sum.
Using \eqref{BC} for $n=0$ and \eqref{IH} with $k=1$ for $n \geq 1$ to bound the $\mcG_{x,v}^1(s)$ and $\mcG_{v}^1(s)$ terms, we find
\begin{equation}
\label{DngI}
I \lesssim s^{1 + |\beta_v|} \|\nabla_x^{n+2}E(s)\|_{\infty}.
\end{equation}

For $n\geq 1$ we estimate $II$ and $III$ using \eqref{IH} to bound derivatives of the field and distribution function.
In particular, for $k \leq n+1$ we use \eqref{IH} to find $\mcD_k^j(t) \lesssim \mcG^k_{x,v}(t) \lesssim 1$ for $j < k$  and  $\mcD_k^j(t) \lesssim \mcG^k_{v}(t) \lesssim 1$ for $j = k$.
This provides a uniform bound on every derivative of $g^\alpha$ appearing within these terms.
The field derivative terms are then estimated using \eqref{IH} in $II$ and $III$, with $k= n+2 - j$ in the latter.
Hence, we find for $s \geq 1$
$$ II \lesssim  \sum_{\substack{|\gamma_x+\gamma_v| = 1 \\ \gamma_x \preceq\beta_x, \gamma_v \preceq\beta_v}} 
s^{|\beta_v - \gamma_v|} s^{-n-4} \ln(s) s$$
and
\begin{align*}
III & \lesssim  \sum_{j=2}^{n}\sum_{\substack{|\gamma_x+\gamma_v| = j \\ \gamma_x \preceq\beta_x, \gamma_v \preceq\beta_v}}
	 s^{|\beta_v-\gamma_v|} s^{-3-n}  \left(s \mcG_{x,v}^{j+1}(s) + \mcG_{v}^{j+1}(s) \right)
    \lesssim \sum_{j=2}^{n}\sum_{\substack{|\gamma_x+\gamma_v| = j \\ \gamma_x \preceq\beta_x, \gamma_v \preceq\beta_v}} s^{-2 + |\beta_v - \gamma_v| -n}.
\end{align*}
Then, as $|\beta_x + \beta_v | = n+2$ and $\gamma_x \preceq \beta_x$, we find 
$$|\beta_v - \gamma_v| = n+2 - |\gamma_v + \gamma_x| - |\beta_x - \gamma_x| \leq   n+2 - |\gamma_v + \gamma_x|.$$
Therefore, we find $|\beta_v-\gamma_v| \leq n+1$ within the estimate of $II$ and $|\beta_v-\gamma_v| \leq n+2 - j$ within the estimate of $III$, which yields
\begin{equation}
\label{DngII}
II  \lesssim s^{-2} \ln(s)
\end{equation}
and
\begin{equation}
\label{DngIII}
III  \lesssim \sum_{j=2}^{n} s^{-j} \lesssim s^{-2},
\end{equation}
respectively.
%
%Further, we note that when $n=0$, we have $II = 0$.
Finally, turning to $IV$, we have
\begin{align}
\label{DngIV}
	IV \lesssim \begin{cases} \displaystyle
		\sum_{\substack{|\gamma_x+\gamma_v| = n+1 \\ \gamma_x \preceq\beta_x, \gamma_v \preceq\beta_v}}	s^{-3+|\beta_v-\gamma_v| -n} \left( s\mcD_{n+2}^{|\gamma_v|}(s)+ \mcD_{n+2}^{|\gamma_v|+1}(s) \right), & n\geq 1 \\
		\displaystyle
		\sum_{\substack{|\gamma_x+\gamma_v| = 1 \\ \gamma_x \preceq\beta_x, \gamma_v \preceq\beta_v}} s^{-4+|\beta_v-\gamma_v|}\ln(s) \left( s\mcD_{2}^{|\gamma_v|}(s)	+ \mcD_{2}^{|\gamma_v|+1}(s) \right), & n=0
	\end{cases}
\end{align}
where  for $n\geq1$ we use \eqref{IH} with $k=1$,
%\begin{align*}
%IV & \lesssim \sum_{\substack{|\gamma_x+\gamma_v| = n+1 \\ \gamma_x \preceq\beta_x, \gamma_v \preceq\beta_v}}
%	s^{-3+|\beta_v-\gamma_v| -n} \left( s\mcD_{n+2}^{|\gamma_v|}(s)
%	+ \mcD_{n+2}^{|\gamma_v|+1}(s) \right).
%\end{align*}
and for $n=0$ we have an improved decay rate due to \eqref{BC}. %of \eqref{IH}. % and this produces
%\begin{align*}
%IV & \lesssim \sum_{\substack{|\gamma_x+\gamma_v| = 1 \\ \gamma_x \preceq\beta_x, \gamma_v \preceq\beta_v}}
%	s^{-4+|\beta_v-\gamma_v|}\ln(s) \left( s\mcD_{2}^{|\gamma_v|}(s)
%	+ \mcD_{2}^{|\gamma_v|+1}(s) \right).
%\end{align*}
%
% Inserting the estimates \eqref{DngI}, \eqref{DngII},  \eqref{DngIII} and \eqref{DngIV} into the inequality \eqref{DngOrig} 
Combining the estimates \labelcref{DngI,DngII,DngIII,DngIV} and inserting them into the integral within \eqref{DngOrig}, we merely integrate $II$ and $III$, as both are integrable on $(1,\infty)$, to find
% 	\begin{align*}
% 		\|D_v^{\beta_v} D_x^{\beta_x} g^\alpha(t)\|_\infty
% 		\lesssim	1 & + \int_1^t  \biggl( s^{1 + |\beta_v|} \|\nabla_x^{n+2}E(s)\|_{\infty} + s^{-2}\ln(s) \\ & \qquad + \sum_{\substack{|\gamma_x+\gamma_v| = n+1 \\ \gamma_x \preceq\beta_x, \gamma_v \preceq\beta_v}}
% 		s^{-3+|\beta_v-\gamma_v| -n} \left( s\mcD_{n+2}^{|\gamma_v|}(s)
% 		+ \mcD_{n+2}^{|\gamma_v|+1}(s) \right)\biggr ) \ ds
% 	\end{align*}	
% or, upon integrating the middle term,
\thickmuskip=2.5mu \medmuskip=2mu \thinmuskip=1.5mu
\begin{equation}
\label{DngGeneral}
	\begin{aligned}
		\|D_v^{\beta_v} D_x^{\beta_x} g^\alpha(t)\|_\infty 
		& \lesssim 1 + \int_1^t \biggl( s^{1 + |\beta_v|} \|\nabla_x^{n+2}E(s)\|_{\infty} + \hspace{-4mm}\sum_{\substack{|\gamma_x+\gamma_v| = n+1 \\ \gamma_x \preceq\beta_x, \gamma_v \preceq\beta_v}} \!\!
		s^{-3+|\beta_v-\gamma_v| -n} \left( s\mcD_{n+2}^{|\gamma_v|}(s)	+ \mcD_{n+2}^{|\gamma_v|+1}(s) \right)\biggr )  ds
	\end{aligned}
\end{equation}
for $n\geq 1$, while for $n = 0$ we have
\begin{equation}
\label{DngGeneraln0}
	\begin{aligned}
		\|D_v^{\beta_v} D_x^{\beta_x} g^\alpha(t)\|_\infty 
		& \lesssim 1 + \int_1^t \biggl( s^{1 + |\beta_v|} \|\nabla_x^{2}E(s)\|_{\infty} + \hspace{-4mm}\sum_{\substack{|\gamma_x+\gamma_v| = 1 \\ \gamma_x \preceq\beta_x, \gamma_v \preceq\beta_v}} \!\!
		s^{-4+|\beta_v-\gamma_v|}\ln(s) \left( s\mcD_{2}^{|\gamma_v|}(s) + \mcD_{2}^{|\gamma_v|+1}(s) \right)\biggr ) ds.
	\end{aligned}
\end{equation}
\thickmuskip=5mu \medmuskip=4mu \thinmuskip=3mu
%Returning to the previous bound, as $\|\nabla^k_x g^\alpha(s) \|_\infty + \|\nabla^k_v g^\alpha(s) \|_\infty\lesssim 1$ for $k \leq n+1$, we use Lemma \ref{LDE} with $k = n+2$ to bound the last terms on the right side, yielding
%\begin{align*}
%	\|\nabla_v^{n+2} g^\alpha(t)\|_\infty
%	& \lesssim	1 + \int_1^t \biggl( s^{-1} \left(1 + \ln^* \left( \Vert \nabla^{n+2}_v g(t) \Vert_\infty \right) \right) + s^{-2 -n} \left( s\mcD_{n+2}^{n+1}(s)
%	+ \mcD_{n+2}^{n+2}(s) \right)\biggr ) \ ds\\
%\end{align*}

Next, we will collect all order $n+2$ derivatives and assemble them with decaying powers of $t$ to construct a Gronwall argument for $\mcD_{n+2}(t)$.
With \eqref{DngGeneral} established, we take the supremum over $\alpha = 1, ..., N$ and $|\beta_x + \beta_v| = n+2$ with $|\beta_v| = j$, multiply by $t^{-j}$, and use $s \leq t$ to find
$$t^{-j}\mcD_{n+2}^j(t)
 \lesssim1 + \int_1^t \biggl( s\|\nabla_x^{n+2}E(s)\|_{\infty}
	+ \sum_{\ell = 0}^{\min\{j,n+1\}}
	s^{-3-\ell -n} \left( s\mcD_{n+2}^{\ell}(s)
	+ \mcD_{n+2}^{\ell+1}(s) \right)\biggr ) \ ds$$
for all $j = 0, ..., n+2$.
Here, we have denoted $\ell = |\gamma_v|$ and further used $|\gamma_v| \leq |\beta_v| = j$ and $|\gamma_v| \leq |\gamma_x + \gamma_v| = n+1$ to rewrite the sum above.
As the terms inside the sum are nonnegative, we further deduce
\begin{equation}
\label{Dnp2}
t^{-j}\mcD_{n+2}^j(t)
 \lesssim1 + \int_1^t \biggl( s\|\nabla_x^{n+2}E(s)\|_{\infty}
	+ \sum_{\ell = 0}^{n+1}
	s^{-3-\ell -n} \left( s\mcD_{n+2}^{\ell}(s)
	+ \mcD_{n+2}^{\ell+1}(s) \right)\biggr ) \ ds
\end{equation}
for all $j = 0, ..., n+2$.
Focusing on the first term within the integrand, we note that $\mcG_v^{n+1}(t) \lesssim 1$ for $n=0$ due to \eqref{BC} and for $n \geq 1$ due to \eqref{IH}. Hence, we use Lemma \ref{LDE} with $k = n+2$, separate the logarithmic product, and use the definition of $\mcD^{n+2}_{n+2}(t)$ so that
\begin{align*}
	\int_1^t s\|\nabla_x^{n+2}E(s)\|_{\infty} ds & \lesssim 1+  \int_1^t  s^{-3-n} \ln^* \left(\max_{\alpha = 1, ..., N} \|\nabla_v^{n+2}g^\alpha(s)\|_{\infty} \right) ds\\
	&\lesssim 1 +  \int_1^t \left( s^{-3-n} \ln^*\left(s^{2+n} \right) + s^{-3-n} \ln^* \left(s^{-2-n} \max_{\alpha = 1, ..., N}\|\nabla_v^{n+2}g^\alpha(s)\|_{\infty} \right) \right) ds \\
	&\lesssim 1+ \int_1^t s^{-3-n} \ln^* \left(s^{-(n+2)} \max_{\alpha = 1, ..., N} \|\nabla_v^{n+2}g^\alpha(s)\|_{\infty} \right) ds \\
	& \lesssim 1 +  \int_1^t s^{-3-n} \ln^* \biggl(s^{-(n+2)} \mcD^{n+2}_{n+2}(s) \biggr ) ds\\
	& \lesssim 1 +  \int_1^t s^{-3-n} \ln^* \biggl(\mcD_{n+2}(s) \biggr ) ds,
\end{align*}
as $\ln^*(x)$ is an increasing function.
Thus, beginning with \eqref{Dnp2}, summing over $j=0,...,n+2$, noting that the right side is independent of $j$, and recalling the definition of $\mcD_{n+2}(t)$, we find
\begin{align*}
	\mcD_{n+2}(t) 	& \lesssim	1 + \int_1^t \biggl( s^{-3-n} \ln^* \biggl(\mcD_{n+2}(s) \biggr ) + s^{-2-n} \sum_{\ell=0}^{n+1}  \left( s^{-\ell}\mcD_{n+2}^{\ell}(s)
	+ s^{-(\ell+1)}\mcD_{n+2}^{\ell+1}(s) \right)\biggr ) \ ds\\
	& \lesssim 1 +  \int_1^t \biggl( s^{-3-n} \ln^* \biggl(\mcD_{n+2}(s) \biggr ) + s^{-2-n}\mcD_{n+2}(s)\biggr ) ds \\
	& \lesssim 1 + \int_1^t s^{-2-n}\ \mcD_{n+2}(s) \ ds.
\end{align*}
We note that this same estimate holds for $n = 0$ as $s^{-4} \ln(s) \lesssim s^{-3}$ within equation \eqref{DngGeneraln0} compared to the estimate of \eqref{DngGeneral}.
Finally, Gronwall's inequality allows us to conclude
$$\mcD_{n+2}(t) \lesssim \exp \left( \int_1^t s^{-2-n} ds \right) \lesssim 1$$
as $n \geq 0$.	

\vspace{0.1in}

\emph{\underline{Step 2:}  Refined estimate of $\infnorm{\nabla_x^{n+2}E(t)}$}

From the definition of $\mcD_{n+2}(t)$, the above estimate implies
\begin{equation}
\label{Dk}
\mcD_{n+2}^\ell(t) \lesssim t^\ell
\end{equation}
for every $\ell = 0, ..., n+2$.
In particular, we find
$$\Vert \nabla_x^{n+2} g^\alpha(t) \Vert_\infty \lesssim \mcD_{n+2}^0(t)  \lesssim 1$$
and
$$\Vert \nabla_v^{n+2} g^\alpha(t) \Vert_\infty \lesssim \mcD_{n+2}^{n+2}(t)  \lesssim t^{n+2},$$
for all $\alpha = 1, ..., N$
which, due to Lemma \ref{LDE} with $k = n+2$, further yields
\begin{equation}
\label{Dnp2E_prelim}
\|\nabla_x^{n+2}E(s)\|_{\infty} \lesssim t^{-4-n} \ln(t)
\end{equation}
for $n \geq 0$.

\vspace{0.1in}

\emph{\underline{Step 3:}  Refined estimates of  $\mcD_{n+2}^\ell(t)$ for $\ell = 1, ..., n+2$}

This field derivative estimate improves the general estimates \eqref{DngGeneral} and \eqref{DngGeneraln0}, yielding
\thickmuskip=3mu \medmuskip=2.75mu \thinmuskip=2mu
	\begin{equation}
	\label{DngGeneral2}
	\begin{aligned}
	\|D_v^{\beta_v} D_x^{\beta_x} g^\alpha(t)\|_\infty
	& \lesssim	1 + \int_1^t \biggl( s^{-3 + |\beta_v| -n} \ln(s)%\\
	% & \quad \ \  
	+ \hspace{-4mm}\sum_{\substack{|\gamma_x+\gamma_v| = n+1 \\ \gamma_x \preceq\beta_x, \gamma_v \preceq\beta_v}}
	\hspace{-0.4cm} s^{-3+|\beta_v-\gamma_v| -n} \left( s\mcD_{n+2}^{|\gamma_v|}(s)
	+ \mcD_{n+2}^{|\gamma_v|+1}(s) \right)\biggr ) ds
	\end{aligned}
	\end{equation}
for $n\geq1$, and for $n=0$
	\begin{equation}
	\label{DngGeneral2n0}
	\begin{aligned}
	\|D_v^{\beta_v} D_x^{\beta_x} g^\alpha(t)\|_\infty
	& \lesssim	1 + \int_1^t \biggl( s^{-3 + |\beta_v|} \ln(s)%\\
	% & \quad \ \ 
	+ \hspace{-4mm} \sum_{\substack{|\gamma_x+\gamma_v| = 1 \\ \gamma_x \preceq\beta_x, \gamma_v \preceq\beta_v}}
	\hspace{-0.3cm} s^{-4+|\beta_v-\gamma_v|} \ln(s) \left( s\mcD_{2}^{|\gamma_v|}(s)
	+ \mcD_{2}^{|\gamma_v|+1}(s) \right)\biggr ) ds.
	\end{aligned}
\end{equation}	
%\thickmuskip=5mu \medmuskip=4mu \thinmuskip=3mu

We now set out to improve \eqref{Dk}, the polynomial growth estimates of $\mcD_{n+2}^\ell(t)$ for $\ell = 1, ..., n+2$. Note that the bound on $\mcD_{n+2}^0(t)$ is already uniform in $t$. We start with the case $n\geq1$. First take $\ell=|\beta_v| = 1$, which implies $|\beta_x| = n+1$ and $|\gamma_v| \leq 1$, and use \eqref{Dk} so that taking the supremeum over all such derivatives in \eqref{DngGeneral2} yields
	\begin{align*}
	\mcD_{n+2}^1(t)
	& \lesssim	1 + \int_1^t \biggl( s^{-2 -n} \ln(s) + s^{-1-n} \mcD_{n+2}^0(s) + s^{-2-n} \mcD_{n+2}^1(s) + s^{-3-n} \mcD_{n+2}^2(s)  \biggr ) \ ds\\
	& \lesssim	1 + \int_1^t \biggl(s^{-1-n} + s^{-2-n} \mcD_{n+2}^1(s)  \biggr ) \ ds.
	\end{align*}	
Upon applying Gronwall and using $n \geq 1$, this gives
$$ \mcD_{n+2}^1(t)  \lesssim \left(1 + \int_1^t s^{-1-n} \ ds \right) \exp \left( \int_1^t s^{-2-n} \ ds \right) \lesssim 1 . $$

Next, with $ \mcD_{n+2}^\ell(t) \lesssim 1$ established for $\ell =1$, we iterate this process over $|\beta_v| = \ell$ with $\ell = 2, ..., n+1$ while assuming that $\mcD_{n+2}^j(t) \lesssim 1$ for all $j = 1, ..., \ell-1$ and again taking the supremum over derivatives within \eqref{DngGeneral2}.
Because $|\beta_v| = \ell$, we find $|\beta_x| = n+2-\ell$ and $|\gamma_v| \leq \ell$, which then implies
%$$ n+1 - \ell \leq n + 1 - |\gamma_v| = |\gamma_x| \leq |\beta_x| \leq n+2-\ell.$$
$$n + 1 - |\gamma_v| = |\gamma_x| \leq |\beta_x| = n+2-\ell$$
within the sum in \eqref{DngGeneral2}.
Rearranging this inequality yields a lower bound on $|\gamma_v|$, namely $|\gamma_v| \geq \ell - 1$.
Combining with the previous upper bound on $|\gamma_v|$, it follows that $\ell - 1 \leq |\gamma_v| \leq \ell$ and only two terms appear within the sum on the right side of \eqref{DngGeneral2}.
Thus, taking the supremum over all derivatives satisfying $|\beta_v| = \ell$ within \eqref{DngGeneral2} and using $\mcD_{n+2}^{\ell-1}(t) \lesssim 1$, we find
\begin{align*}
	\mcD_{n+2}^\ell(t)
	& \lesssim	1 + \int_1^t \biggl[ s^{-3 + \ell -n} \ln(s) + \sum_{j=\ell-1}^\ell s^{-3+\ell-j-n} \left( s\mcD_{n+2}^j(s) + \mcD_{n+2}^{j+1}(s) \right) \biggr ] \ ds\\
	& \lesssim	1 + \int_1^t \biggl[ s^{-2}\ln(s) + s^{-1-n} \mcD_{n+2}^{\ell -1}(s) + s^{-2-n} \mcD_{n+2}^{\ell}(s) + s^{-3-n} \mcD_{n+2}^{\ell+1}(s) \biggr ] \ ds\\
	& \lesssim	1 + \int_1^t \biggl[ s^{-1-n} + s^{-2-n} \mcD_{n+2}^\ell(s) + s^{-3-n} \mcD_{n+2}^{\ell +1}(s) \biggr ] \ ds
\end{align*}	
Because $n \geq 1$, the first term is integrable, and applying Gronwall yields
\begin{equation}
\label{Dnp2ell}
\mcD_{n+2}^\ell(t)  \lesssim \left(1 + \int_1^t  s^{-3-n} \mcD_{n+2}^{\ell +1}(s) \ ds \right) \exp \left( \int_1^t s^{-2-n} \ ds \right) \lesssim 1 + \int_1^t  s^{-3-n} \mcD_{n+2}^{\ell +1}(s) \ ds.
\end{equation}
Using \eqref{Dk} with $k = \ell+1$, we find 
$$\mcD_{n+2}^\ell(t)  \lesssim 1 + \int_1^t  s^{-2-n+\ell} \ ds$$
and conclude both
$\mcD_{n+2}^\ell(t) \lesssim 1$
for $\ell = 2, ..., n$ and
$\mcD_{n+2}^{n+1}(t) \lesssim \ln(t)$.
The latter estimate will be improved to a constant bound at the conclusion of the proof.
	
To improve the growth estimates for the case $n = 0$, we first note that \eqref{Dk} with $k=0$ yields $\mcD_{2}^0(t) \lesssim 1$ and with $k=2$ gives $\mcD_{2}^2(t) \lesssim t^2$. Hence, repeating the above steps using \eqref{DngGeneral2n0} instead of \eqref{DngGeneral2}, we arrive at
	\begin{align*}
	\mcD_{2}^1(t)
	& \lesssim	1 + \int_1^t \biggl( s^{-2} \ln(s) + s^{-2} \ln(s) \mcD_{2}^0(s) + s^{-3} \ln(s) \mcD_{2}^1(s) + s^{-4}\ln(s) \mcD_{2}^2(s)  \biggr ) \ ds\\
	& \lesssim	1 + \int_1^t \biggl(s^{-2}\ln(s) + s^{-3}\ln(s) \mcD_{2}^1(s)  \biggr ) \ ds.
	\end{align*}	
Upon applying Gronwall, this gives
$$ \mcD_{2}^1(t)  \lesssim \left(1 + \int_1^t s^{-2}\ln(s) \ ds \right) \exp \left( \int_1^t s^{-3}\ln(s) \ ds \right) \lesssim 1.$$
Hence, combining this bound for $\mcD_{2}^0(t)$ and $\mcD_{2}^1(t)$ results in
$\mcG_{x,v}^{2}(t) \lesssim 1.$

It remains to improve the estimate of $\mcD_{n+2}^{n+2}(t)$ and $\mcD_{2}^{2}(t)$ using another Gronwall argument, as follows. We start with $n\geq1$. We repeat the steps that led to \eqref{DngGeneral}, but without upper bounding the derivatives of $g^\alpha$ by the $\mcD_{n+2}^j(t)$ terms in $IV$, and consider the case of $|\beta_x| = 0$, so that $|\beta_v| = n+2$.
Then, $|\gamma_x| = 0$ and $|\gamma_v| = n+1$ necessarily, which gives
	\begin{equation}
	\label{Dvnp2g}
	\begin{aligned}
	\|\nabla_v^{n+2} g^\alpha(t)\|_\infty
	 & \lesssim 1 + \int_1^t \biggl[ s^{n+3} \|\nabla_x^{n+2}E(s)\|_{\infty}\\
	 & \hspace{2cm} + s\Vert \nabla_x E(s) \Vert_\infty \Big( s\Vert \nabla_x \nabla_v^{n+1} g^\alpha(s) \Vert_\infty + \Vert \nabla_v^{n+2} g^\alpha(s) \Vert_\infty \Big)\biggr ] \ ds.
	\end{aligned}
	\end{equation}
%which, upon taking supremums, yields \eqref{DngDnE}.
%$$\mcG_v^{n+2}(t) \lesssim 1 + \int_1^t \biggl( s^{n+3} \|D_x^{n+2}E(s)\|_{\infty} + s^{-2 -n} \left( s\Vert \nabla_x \nabla_v^{n+1} g^\alpha(s) \Vert_\infty + \mcG_v^{n+2}(s) \right)\biggr ) \ ds.$$
%
Using \eqref{IH} with $k=1$ and $\mcD_{n+2}^{n+1}(t) \lesssim \ln(t)$ within this estimate gives 
$$\mcG_v^{n+2}(t) \lesssim 1 + \int_1^t \biggl[ s^{n+3} \Vert \nabla_x^{n+2} E(s) \Vert_\infty + s^{-1-n}\ln(s)  + s^{-2-n}\mcG_v^{n+2}(s) \biggr ] \ ds.$$
As the middle term is integrable for $n \geq 1$, we apply Gronwall's inequality to find
$$\mcG_v^{n+2}(t)  \lesssim \left(1 + \int_1^t s^{n+3} \Vert \nabla_x^{n+2} E(s) \Vert_\infty \ ds \right) \exp \left( \int_1^t s^{-2-n} \ ds \right)  \lesssim 1 + \int_1^t s^{n+3} \Vert \nabla_x^{n+2} E(s) \Vert_\infty \ ds.$$
For $n=0$ we use \eqref{BC} instead of \eqref{IH} within \eqref{Dvnp2g}, as well as $\mcG_{x,v}^{2}(t) \lesssim 1$, which yields
$$\mcG_v^{2}(t) \lesssim 1 + \int_1^t \biggl[ s^{3} \Vert \nabla_x^{2} E(s) \Vert_\infty + s^{-2}\ln(s)  + s^{-3}\ln(s)\mcG_v^{2}(s) \biggr ] \ ds.$$
Again, the middle term is integrable, and an application of Gronwall gives
$$\mcG_v^{2}(t)  \lesssim 1 + \int_1^t s^{3} \Vert \nabla_x^{2} E(s) \Vert_\infty \ ds.$$
Thus, we have obtained \eqref{DngDnE0} in the $n=0$ case and \eqref{DngDnE} for $n \geq 1$.
Using \eqref{Dnp2E_prelim} in \eqref{DngDnE0} and \eqref{DngDnE} results in
$$\mcG_v^{n+2}(t) \lesssim 1 + \int_1^t s^{-1}\ln(s) \ ds  \lesssim \ln^2(t)$$
for all $n\geq0$.
Finally, using this improved bound on $\mcG_v^{n+2}(t)$, we have $\mcD_{n+2}^{n+2}(t) \lesssim \mcG_v^{n+2}(t) \lesssim \ln^2(t)$. Inserting this bound within \eqref{Dnp2ell} with $\ell = n+1$, we find $\mcD_{n+2}^{n+1}(t) \lesssim 1$, and hence
$\mcD_{n+2}^\ell(t) \lesssim 1$ for $\ell = 1, ..., n+1$.
Therefore,
$\mcG_{x,v}^{n+2}(t) \lesssim 1$ for $n\geq1,$
and the proof is complete.
\end{proof}

%\section{Ackowledgements}
%The author would like to thank the anonymous reviewers for their valuable comments and suggestions to improve the paper.

\end{document}